\documentclass[12pt]{amsart}
\usepackage{amsmath,amsthm,indentfirst}
\usepackage{graphicx}
\usepackage{amssymb}
\usepackage{amsfonts}
\usepackage{a4}
\usepackage{amscd}
\usepackage[latin1]{inputenc}
\usepackage{url}
\DeclareMathAlphabet{\mathpzc}{OT1}{pzc}{m}{it}
\theoremstyle{plain}
\newtheorem{thm}{Theorem}
\newtheorem*{maintheorem*}{Main Theorem}
\newtheorem*{thm*}{Theorem}
\newtheorem*{thma*}{Theorem A}
\newtheorem*{thmaa*}{Theorem A'}
\newtheorem*{thmb*}{Theorem B}
\newtheorem*{thmo*}{Theorem 1.1}
\newtheorem*{thmc*}{Theorem C}
\newtheorem*{thmd*}{Theorem D}
\newtheorem*{thmf*}{Theorem 4.1}
\newtheorem*{conjecture*}{Conjecture}
\newtheorem*{prop*}{Proposition}
\newtheorem{cor}[thm]{Corollary} 
\newtheorem{lem}[thm]{Lemma}
\newtheorem{prop}[thm]{Proposition}

\newtheorem{definition}[thm]{Definition}

\newtheorem{remark}[thm]{Remark}

\def\bbz{\mathbb{Z}}
\def\bbq{\mathbb{Q}}

\def\bbr{\mathbb{R}}
\def\bba{\mathbb{A}}

\def\bbc{\mathbb{C}}

\def\bbh{\mathbb{H}}
\def\bbn{\mathbb{N}}
\def\bbg{\mathbb{G}}
\def\bbt{\mathbb{T}}
\def\bbv{\mathbb{V}}
\def\bbu{\mathbb{U}}
\def\bbp{\mathbb{P}}
\def\bbl{\mathbb{L}}
\def\bbs{\mathbb{S}}

\def\bbm{\mathbb{M}}
\def\bbx{\mathbb{X}}

\def\ucal{\mathcal{U}}
\def\vcal{\mathcal{V}}
\def\ocal{\mathcal{O}}

\def\ccal{\mathcal{C}}

\def\mcal{\mathcal{M}}

\def\Rcal{\mathcal{R}}

\def\pcal{\mathcal{P}}
\def\wcal{\mathcal{W}}
\def\afr{\mathfrak{a}}

\def\ufr{\mathfrak{u}}

\def\gfr{\mathfrak{g}}

\def\ybf{\mathbf{y}}
\def\xbf{\mathbf{x}}

\def\vbf{\mathbf{v}}
\def\1bf{\mathbf{1}}
\def\ebf{\mathbf{e}}
\def\apzc{\mathpzc{a}}
\def\bpzc{\mathpzc{b}}
\DeclareMathOperator\vol{vol}
\DeclareMathOperator\SL{SL}
\DeclareMathOperator\GL{GL}
\DeclareMathOperator\PSL{PSL}

\DeclareMathOperator\Ad{Ad}
\DeclareMathOperator\Lie{Lie}

\DeclareMathOperator\diag{diag}
\DeclareMathOperator\supp{supp}
\newcommand{\leftexp}[2]{{\vphantom{#2}}^{#1}{#2}}
\def\h{\hspace{1mm}}

\def\vare{\varepsilon}
\def\be{\begin{equation}}
\def\ee{\end{equation}}
\def\ba{\backslash}
\def\TP{\widetilde{\Phi}}
\def\TD{\widetilde{\Delta}}
\newcommand{\wt}[1]{\widetilde{#1}}
\def\peone{P_E^{(1)}}
\def\pfone{P_F^{(1)}}

\begin{document}
\title{Translate of horospheres and counting problems}
\author{Amir Mohammadi}
\address{Dept of Mathematics,University of Texas at Austin, 1 University Station, C1200, Austin, TX 78750}
\email{amir@math.utexas.edu}
\thanks{A. M. was partially supported by the NSF grant DMS-1200388.}
\author{Alireza Salehi Golsefidy}
\address{Mathematics Dept, University of California, San Diego, CA 92093-0112}
\email{golsefidy@ucsd.edu}
\thanks{A. S-G. was partially supported by the NSF grant DMS-0635607 and NSF grant DMS-1001598.}
\subjclass{22E40}
\date{}
\begin{abstract}
Let $G$ be a semisimple Lie group without compact factors, $\Gamma$ be an irreducible lattice in $G.$  
In the first part of the article we give the necessary and sufficient condition under which a sequence of translates of probability ``horospherical measures" 
is convergent. And the limiting measure is also determined when it is convergent (see Theorems~\ref{trans} and~\ref{thm;trans-abs} for the precise statements). 
In the second part, two applications are presented. The first one is of geometric nature and the second one gives an alternative way to count the number of rational points on a flag variety. 
\end{abstract}
\maketitle
\section{Introduction.}
The equidistribution properties of horospherical orbits have been studied  extensively. Various methods have been employed to study this kind of problem, 
Hedlund proved minimality and Furstenberg proved the unique ergodicity of horocycle flows on compact surfaces. S.~Dani~\cite{D1, D2} classified all ergodic invariant measures for the action of a maximal horospherical subgroup of a semisimple Lie group. He also classified closure of the orbits of (not necessarily maximal) horospherical subgroups in~\cite{D3}. His method, to classify measures, relies ultimately on Furstenberg's idea for the proof of the unique ergodicity of horocycle flow on compact surfaces. Another possible approach to this problem would be to use representation theory, see the work of M.~Burger in~\cite{Bur} where these lines of ideas are utilized to classify invariant Radon measures for horocycle flows on geometrically finite surfaces. Also mixing properties of semisimple elements can be used to obtain measure classification for the action of the corresponding horospherical subgroups. These lines of ideas go back to G. A. Margulis' PhD thesis~\cite{Mar4}. They have been extensively utilized by many others ever since, see~\cite{Rat1} where similar ideas are present.  

As we mentioned horocycle flows on the unit tangent bundle of a hyperbolic surface $M$ of finite volume was the starting point of the study of such dynamics. In fact, in this case, Sarnak~\cite{Sar} gave the precise asymptotic behavior of closed horocycle orbits using Eisenstein series. If we do not ask for the best possible error rate, then (as we said earlier) one can describe such an asymptotic behavior using the mixing property of the geodesic flow. The latter has this advantage that it can be easily extended to any rank-1 finite volume locally symmetric manifold. The main reason being that the {\it entire} horospherical orbit gets expanded or contracted depending on the direction of the flow on an orthogonal geodesic. When $G$ is a higher rank semisimple Lie group and $\Gamma$ is a lattice in $G$, then for any closed horospherical orbit $\pi(Ux_0)$ in $G/\Gamma$ there are geodesic flows which expand some part of $\pi(Ux_0)$ and contract some 
other parts. That is the main reason why most of previous works only deal with 
translations in the expanding directions (the directions which expand {\it all} the horospherical orbit.) \cite{Sh,KM1, KW}.\footnote{The study of an arbitrary unipotent flow is much harder. In 
a series of seminal papers Ratner~\cite{Rat2,Rat3,Rat4} gave a complete description of probability measures on a homogeneous space
which are invariant under an arbitrary unipotent flow and then showed the 
closure of any orbit of a unipotent flow is homogeneous.}

In this work, we use soft methods to give the necessary and sufficient 
condition under which a sequence of translations of a closed horospherical 
orbit is equidistributed in the limit in a homogeneous closed subset\footnote{In 
a personal communication Peter Sarnak told us that this result can be also 
proved using Eisenstein series.}. Then using Ratner's theorem we show that
translations of a measure which is absolutely continuous with respect to a
horospherical measure have similar properties. In the second part of the 
article, two applications are given. The first one is an extension of \cite[Theorem 7.1]{EM}. We start with a finite volume locally symmetric space
$\mcal=K\ba G/\Gamma$ and consider a closed horosphere 
$\overline{\ucal}=K\ba Kg_0U\Gamma/\Gamma$ in $\mcal$. We count the number of the 
lifts of $\overline{\ucal}$ in the symmetric space $K\ba G$ which intersect a 
ball of radius $R.$\footnote{When $G=\PSL_2(\bbr)$, using mixing of geodesic 
flows, this result is proved in \cite[Theorem 7.1]{EM} and using the same idea 
one can easily get a similar result for any real rank-1 group. But this method 
is not good enough in the higher rank case.} As another application, we count number 
of rational points on a generalized flag variety with respect to any line-bundle, 
generalizing \cite{FMT} (the case of the anti-canonical line-bundle), and \cite{BaT} and~\cite{Pe} (for the case of an arbitrary metrized line-bundle, with split and quasi-split groups).\footnote{Eisenstein series is the main ingredient in \cite{FMT, BaT, Pe}.} 

Our strategy to handle the counting problems using dynamics is a well developed technique which originated in~\cite{DRS}, see also~\cite{EM}, and has been utilized by many authors, see~\cite{Oh}
and references therein.
The problems in hand, however, present one notably different feature: in almost all previous works one deals with cases in which the possibility of ``intermediate subgroups" is eliminated by the nature of the problem, see e.g. the ``non-focusing" condition in~\cite{EMS2}. On the other hand, intermediate behaviors naturally occur in problems we deal with. Thus, for the applications considered in this paper, a careful analysis of this feature is needed which adds new difficulties in carrying out the proofs.
\section{Statement of the results}\label{s:statement}
We need to set a few notations before stating the main results of this work. 
Let $\bbg$ be a simply connected almost simple $\bbq$-group with a 
fixed embedding in $\mathbb{GL}_l$. 
Our standing assumption in this paper is that $\bbg$ is $\bbq$-isotropic;
indeed the analysis carried out in this paper is only non-trivial under this condition. 
Let $\Gamma=\bbg(\bbz)=\bbg(\bbq)\cap\GL_l(\bbz)$. Let $\bbs$ 
be a maximal $\bbq$-split $\bbq$-torus. Let $\bbs\subseteq \bbp$ be a minimal $\bbq$-parabolic 
subgroup and let $\Delta$ and $\{\lambda_{\alpha}\}_{\alpha\in \Delta}$ be the associated simple roots relative 
to $\bbs$ and the  fundamental $\bbq$-weights (see \cite[Section 12]{BT}), respectively. For any $E\subseteq \Delta$, let $\bbp_E$ be the associated standard $\bbq$-parabolic subgroup, e.g. $\bbp_{\varnothing}=\bbp$ and $\bbp_{\Delta}=\bbg,$ in particular, for every $\alpha\notin E$ the fundamental weight $\lambda_\alpha$ defines a character of $\bbp_E$. 

Let $G=\bbg(\bbr)$, $A=\bbs(\bbr)^{\circ}$, let $Q_E=\overline{\bbp_E(\bbr)^+(\Gamma\cap\bbp_E(\bbr))}^\circ,$ where $\bbp_E(\bbr)^+$ is the group generated by $\bbr$-unipotent subgroups of $\bbp_E(\bbr)$ and the closure is taken with respect to the usual topology.

Alternatively $Q_E$ maybe described as 
follows: let $\bbp_E=\bbl_E R_u(\bbp_E),$ where $\bbl_E$ is the standard Levi factor. Let $\bbl_E=Z(\bbl_E)[\bbl_E,\bbl_E].$ Put $\bbm_E'=[\bbl_E,\bbl_E];$ this is a semisimple $\bbq$-group. 
Let $\bbm'_E=\bbh_1\cdots\bbh_k$ be almost direct product of $\bbq$-simple groups. Put 
\[
\mbox{ $M_E=\prod_{i\in I} \bbh_j(\bbr)^\circ,\quad$
 where $I=\{1\leq i\leq k: \bbh_i\;\text{is $\bbr$-isotropic}\}$}.
\]
Then, $Q_E=M_ER_u(\bbp_E(\bbr)).$
In particular, since $\bbg$ is simply connected, $G$ is the product of almost simple factors, and thanks to the fact that $\bbg$ is $\bbq$-isotopic $Q_{\Delta}=G.$ 

Let us mention some of the significant properties of this subgroup
\begin{enumerate}
 \item $Q_E\cap\Gamma$ is a lattice in $Q_E.$
 \item $\bbp_E(\bbr)^+$ acts ergodically on $Q_E/Q_E\cap\Gamma.$
 \item For any $E\subseteq \Delta$ let $\bbp_E^+\subseteq\bbp_E$ be the $\bbq$-subgroup generated by $\bbq$-unipotent subgroups. Then, 
 $Q_E$ is the group generated by $\bbp_E^+(\bbr)$ and $Q_\varnothing$, see the proof of Proposition~\ref{p:FindTheLimitUTrans}.
\end{enumerate}

For any topological space $X$, let $\pcal(X)$ be the space of probability 
measures on $X$ equipped with the weak* topology. If $H$ is a closed subgroup of $G$ such that $H\cap\Gamma$ is a 
lattice in $H$, then $\mu_H\in \pcal(G/\Gamma)$ denotes the $H$-invariant 
probability measure supported on $H\Gamma/\Gamma$. For any probability measure $\mu\in \pcal(X)$ and $g\in G$, $g\mu$ is a probability measure on $X$ such that $(g\mu)(E):=\mu(g^{-1}E)$; in particular $\supp(g\mu)=g\supp(\mu)$. 

\begin{thm}~\label{trans} 
In the above notation, let $\{a_n\}\subseteq A$ and $E\subseteq \Delta$.  Then
\begin{itemize}
\item[i)] If $\lambda_{\alpha} (a_n)$ goes to zero for some $\alpha\not \in E$, then $a_n \mu_{Q_E}$ diverges in $\pcal (G/\Gamma)$.
\item[ii)] Let $E\subseteq F\subseteq \Delta$. If $\lambda_{\alpha}(a_n)=1$ for any $n$ and any $\alpha\not \in F$ and $\lambda_{\alpha}(a_n)$ goes to infinity for any $\alpha\in F\setminus E$, then $a_n \mu_{Q_E}$ converges to $\mu_{Q_F}$. In particular, when $F=\Delta$, it converges to the Haar measure on $G/\Gamma$.
\end{itemize}
\end{thm}
In Section~\ref{ss;Abs-cont} we combine Theorem~\ref{trans} and 
Ratner's classification of measures invariant under a unipotent subgroup 
to get a similar statement for translates of certain absolutely continuous 
measures with respect to $\mu_{Q_E}.$ The following is a special 
case of what is proved there. See Theorem~\ref{thm;trans-abs-general} 
and the discussion proceeding that theorem for a general statement and 
a discussion of extra conditions imposed in the absolutely continuous case.  

\begin{thm}\label{thm;trans-abs}
Let the notation and assumptions be as in Theorem~\ref{trans} and let $\mu=f\mu_{Q_E}\in\pcal(G/\Gamma)$
where $f$ is a non-negative continuous function on $G/\Gamma.$
Then
\begin{itemize}
\item[i)] If $\lambda_{\alpha} (a_n)$ goes to zero for some $\alpha\not\in E$, then $a_n \mu$ 
diverges in $\pcal (G/\Gamma)$.
\item[ii)] Let $\{a_n\}\subseteq A$ be a sequence so that
$\lambda_{\alpha}(a_n)$ goes to infinity for any $\alpha\in\Delta\setminus E .$ Then
$a_n\mu$ converges to the $G$-invariant Haar measure on $G/\Gamma$.
\end{itemize}  
\end{thm} 

Let us give a brief account on Theorem~\ref{trans}.
In view of Theorem~\ref{trans} the limiting behavior of the translates of the measures in question 
is governed by the fundamental weights and not the fundamental roots. 
For instance this means that, if $\pi(Q_{\varnothing})$ is translated along a ray 
in the interior of the {\em dual cone} of the positive Weyl chamber, 
it gets equidistributed in $G/\Gamma$ with respect to the Haar measure. 
This cone is larger than the positive Weyl in higher rank. 
However, the dual of a one dimensional cone is in a finite Hausdorff distance of itself, 
that is why the rank one case is much easier to handle. 

To further put this in perspective let us consider the explicit example
$\bbg=\mathbb{SL}_3$. Then we can and will assume that $\bbs$ is the 
$\bbq$-subgroup given by diagonal matrices with determinant one. 
Then $G=\SL_3(\bbr)$, $\Gamma=\SL_3(\bbz)$, and 
\[
A=\{\diag(a_1,a_2,a_3)|\h a_1,a_2,a_3\in \bbr^+,\h a_1a_2a_3=1\}.
\]
 Then there are six roots: for any $1\le i, j\le 3$ and $i\neq j$, 
 \[
 \phi_{ij}(\diag(a_1,a_2,a_3))=a_ia_j^{-1}.
 \]
  We can assume that $\Delta=\{\alpha_1, \alpha_2\}$, where $\alpha_1(\diag(a_1,a_2,a_3))=a_1a_2^{-1}$ 
  and $\alpha_2(\diag(a_1,a_2,a_3))=a_2a_3^{-1}$. Then 
\begin{enumerate}
	\item The coroots are $\alpha_1^{\vee}(a):=\diag(a,a^{-1},1)$ and $\alpha_2^{\vee}(a):=\diag(1,a,a^{-1})$.
	\item The fundamental weights are 
$\lambda_1(\diag(a_1,a_2,a_3)):=a_1$ and $\lambda_2(\diag(a_1,a_2,a_3)):=a_1a_2$ (note that $\lambda_i\circ\alpha_j^{\vee}(a)=a^{\delta_{ij}}$.).
  \item The corresponding proper parabolic subgroups are 
\[
 \bbp_\varnothing=\left(\begin{array}{ccc}* & * & *\\ 0 & * & *\\ 0 & 0 & *\end{array}\right),\quad
 \bbp_{\alpha_1}=\left(\begin{array}{ccc}* & * & *\\ * & * & *\\ 0 & 0 & *\end{array}\right),\quad
 \bbp_{\alpha_2}=\left(\begin{array}{ccc}* & * & *\\ 0 & * & *\\ 0 & * & *\end{array}\right).
 \]
  \item The corresponding $Q$ subgroups are 
\[
 Q_\varnothing=\left(\begin{array}{ccc}1 & * &*\\ 0 & 1& *\\ 0& 0&1 \end{array}\right),\quad
 Q_{1}=\left(\begin{array}{cc}\SL_2(\bbr) & *\\ 0 & 1\end{array}\right),\quad
 Q_{2}=\left(\begin{array}{ccc}1 & * \\ 0 & \SL_2(\bbr) \end{array}\right).
 \]
\end{enumerate}  
We view the group of characters $X^*(\bbs)$ and co-characters $X_*(\bbs)$ as additive groups. Note that $\chi\circ\theta(x)=x^{\langle \chi,\theta\rangle}$ gives us a non-degenerate bilinear form $\langle,\rangle:X^*(\bbs)\times X_*(\bbs)\rightarrow \bbz$. The Weyl
group (which, in this case, is the group of permutations) acts linearly on $X^*(\bbs)$. We take a Euclidean structure on $E:=X^*(\bbs)\otimes_{\bbz} \bbr$ which is invariant under the Weyl group. We can and will identify three 
Euclidean spaces: $E$, $X_*(\bbs)\otimes_{\bbz}\bbr$ and $\afr:=\Lie(A)$ (the Killing form induces a Euclidean structure on $\afr$). In particular, for any $\chi\in X^*(\bbs)$ and $a\in A$, we have $\chi(a)=e^{(\chi,\log(a))}$, where $\chi$ and $\log(a)$ are realized as elements of $E$. In Figure~\ref{F:Example}, the positive Weyl chamber (the blue shaded region) and the cone given by Theorem~\ref{trans} (the red shaded region) are given: in the sense that, if $\pi(Q_{\varnothing})$ is translated along the exponent of a ray in the interior of the red region, then it gets equidistributed in $G/\Gamma$ with respect to the Haar measure. And if $\pi(Q_{\varnothing})$ is translated along the walls of the red region, then it gets equidistributed with respect to the Haar measure of the corresponding $Q_E$ group. Along a ray in the interior of the complement of this cone, $\pi(Q_{\varnothing})$ escapes to infinity.   

\begin{figure}
\includegraphics[width=\columnwidth]{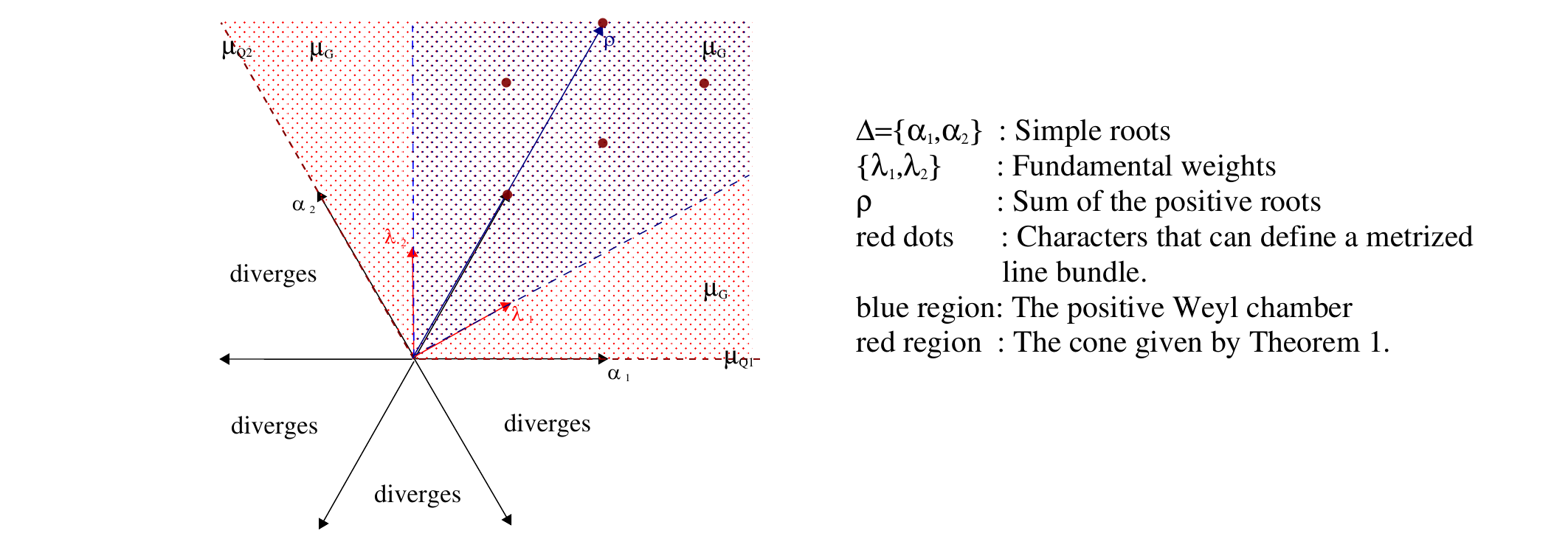}
\caption{The $\mathbb{SL}_3$ case.}
\label{F:Example}
\end{figure} 

So Theorem~\ref{trans} implies that 
\begin{enumerate}
	\item If either $\{a_n^{(1)}\}$ or $\{a_n^{(1)}a_n^{(2)}\}$ goes zero, then $\{\diag(a_n^{(1)},a_n^{(2)},a_n^{(3)})\mu_{Q_{\varnothing}}\}$ diverges.
	\item For $i=1$ or $2$, we have $\alpha_i^{\vee}(a)\mu_{Q_{\varnothing}}\rightarrow \mu_{Q_{i}}$ as $a$ goes to infinity.
	\item If $a_n^{(1)}$ and $a_n^{(1)}a_n^{(2)}$ go to infinity (as $n$ goes to infinity), then $\diag(a_n^{(1)},a_n^{(2)},a_n^{(3)})\mu_{Q_{\varnothing}}\rightarrow \mu_G$.	
\end{enumerate}
As one can see in Figure~\ref{F:Example} there are rays $\{a_t\}$ in $A$ which are outside of the positive Weyl chamber, but at the same time $a_t\mu_{Q_{\varnothing}}$ converges to $\mu_G$. Indeed if $\{a_t\}$ is a ray in the interior of the positive Weyl chamber, then using the mixing property of the action of $A$ and ``thickening'' of $U=Q_{\varnothing}$ one could easily deduce the desired result (this is called {\em wavefront phenomenon}). In fact, if the rank is large enough, then there are co-characters $\theta\in X_*(\bbs)$ such that  
\begin{enumerate}
	\item If $\bbp^+(\theta)$ is the associated parabolic subgroup, then $R_u(\bbp^+(\theta))\cap \bbp_{\varnothing}$ does not contain $R_u(\bbp_E)$ for any $E\subseteq \Delta$. 
	\item $\theta(t)\mu_{Q_{\varnothing}}\rightarrow \mu_G$ as $t$ goes to infinity. 
\end{enumerate}
These examples show that one cannot use the wavefront phenomenon for unipotent radical of larger parabolic subgroups to get the desired result. For instance, 
let $\bbg=\mathbb{SL}_5$ and $\bbs$ be the subgroup given by the diagonal matrices. Let
\[
\theta(t):=\diag(t^6,t^7,t^{-12},t^9,t^{-10}).
\]
Then here are the expansion and contraction directions as $t$ goes to infinity (the $ij$ component is the sign of $\langle \phi_{ij}, \theta\rangle$ for $i\neq j$):
\[
\left[\begin{array}{ccccc}
0&-&+&-&+\\
+&0&+&-&+\\
-&-&0&-&-\\
+&+&+&0&+\\
-&-&+&-&0	
\end{array}\right]. 
\] 
So clearly $R_u(\bbp^+(\theta))\cap \bbp_{\varnothing}$ does not contain $R_u(\bbp_E)$ for any $E\subseteq \Delta$. On the other hand, we have
\[
\langle \lambda_1,\theta\rangle= 6,\h \langle \lambda_2,\theta\rangle= 13,\h \langle \lambda_3,\theta\rangle= 1,\h \langle \lambda_4,\theta\rangle= 10,
\]
and so by Theorem~\ref{trans} we have that  $\theta(t)\mu_{Q_{\varnothing}}\rightarrow \mu_G$ as $t$ goes to infinity ($\phi_{ij}$, $\alpha_i$ and $\lambda_i$ are defined as in the case of $\bbg=\mathbb{SL}_3$). 

Though the above example shows that conjugating along a ray $\{a_t\}$ in the {\em correct} cone does not necessarily expand all the root directions of $R_u(\bbp_E)$ for any $E\subseteq \Delta$, but it does send the volume form of 
$R_u(\bbp_E)$ to infinity for any $E\subseteq \Delta$.

{\bf A geometric application:} as it is mentioned earlier, the first application of 
Theorem~\ref{trans} is a generalization of \cite[Theorem 7.1]{EM}. 
Let $\mcal$ be a finite volume locally symmetric irreducible Riemannian orbifold. That is: there is a connected semisimple Lie group $G$ and an irreducible lattice $\Gamma$ in $G$ such that $X\simeq K\ba G$, where $K$ is a maximal compact subgroup of $G$ and $x_0$ is mapped to the identity coset in $X$. Moreover, $\pi:X\rightarrow X/\Gamma\simeq \mcal$ gives the universal map. Suppose there exists a maximal unipotent subgroup $U$ of $G$ and some $g_0\in G$ such that $\overline{\ucal}=\pi(K\ba Kg_0U)$ is closed in $\mcal.$
Put $\ucal=K\ba Kg_0U.$

We count the number of lifts of $\overline{\ucal}$ to the universal 
cover $X$ which intersect the ball $B(x_0,R)$ centered at $x_0$ and of 
radius $R$ (goes to infinity). This in particular is interesting when the (real) rank of $G$ is at least 2 (also known as the higher rank case) as the method in \cite[Theorem 7.1]{EM} no longer works. Hence we will assume that we are in higher rank case. Thus,  
by Margulis' arithmeticity theorem~\cite[Chapter IX]{Mar-book}, 
there is an absolutely almost simple simply connected $\bbq$-group $\bbg$ such that
\begin{enumerate}
	\item There is a surjective homomorphism $\iota$ from $\bbg(\bbr)$ to $G$ with a finite central kernel.
	\item $\Gamma$ is commensurable with $\iota(\bbg(\bbz))$.
	\item There is a minimal $\bbq$-parabolic $\bbp$ and $g_0\in G$ such that 
	$\overline{\ucal}=\pi(K\ba Kg_0\iota(U'))$, where $U'=R_u(\bbp)(\bbr)$.
\end{enumerate} 
Before stating our result, let us introduce a few other notation and refer the reader to Section~\ref{ss:Riemannian} for more details. As we explained earlier
the Killing form induces a Euclidean structure on $\afr:=\Lie(A)$ and we identify it with $E:=X^*(\bbs)\otimes_{\bbz}\bbr$ and $X_*(\bbs)\otimes_{\bbz}\bbr$. And we choose a Riemannian metric on $X$ such that $d(x_0,a x_0)=\|\log(a)\|$ for any $a\in A$. Let $\rho_{\Delta}\in X^*(\bbs)$ be
\[
\rho_{\Delta}(a):=\det(\Ad(a)|_{\Lie(R_u(\bbp_{\varnothing})(\bbr))}),
\]
and realize it as a vector in $E$.

\begin{thm}\label{t:CountingLifts}
As $R$ goes to infinity,
\[
\#\hspace{.5mm}\{\gamma\in\Gamma\hspace{1mm}|\hspace{1mm}\mathcal{U}\gamma\cap B(x_0,R)\neq\varnothing\} \sim \frac{\vol(U/U\cap\Gamma)}{\vol(\mathcal{M})}\cdot  \rho_{\Delta}(a_0)\cdot\left(\frac{2\pi R}{\|\rho_{\Delta}\|}\right)^{\frac{r-1}{2}}\cdot e^{\|\rho_{\Delta}\| R},
\]
where $\vol$ are Riemannian volume forms, $a_0\in A$ such that $Kg_0U=Ka_0U$ and $r$ is the $\bbr$-rank of $G$. 
\end{thm}
It is worth mentioning that in the setting of Theorem~\ref{t:CountingLifts} the $\bbq$-rank and the $\bbr$-rank
of $\bbg$ are equal. This is due to the fact that $U,$ the maximal unipotent subgroup of $G,$
has a closed orbit. Therefore, in this case the Iwasawa decomposition of $G$ gives $G=KAU$.

{\bf Rational points on a flag variety:} we also count the number of rational points on a flag variety $\bbx$ with 
respect to an arbitrary metrized line bundle. Any (generalized) flag variety is of the form
$\bbg/\bbp_E$ for some $E\subseteq \Delta$. Any $\bbq$-character $\chi$ of 
$\bbp_E$ induces the following line-bundle on $\bbx$: $\bbl_{\chi}:=\bbg\times \bbg_a/\sim$, where $(g,x)\sim (gp,\chi(p)x)$ for any $p\in \bbp_E$. In fact, 
since $\bbg$ is simply connected, any line-bundle is of this 
form~\cite[Section 6.2.1]{Pe}.  
Moreover, when $\bbg$ is simply-connected, $\bbl_{\chi}$ is 
a metrized line-bundle if and only if there are positive integers $n_\alpha$ such 
that the restriction of $\chi$ to $\bbs$ is equal to $\sum_{\alpha\not\in E} n_\alpha \lambda_\alpha,$ we are viewing the group of characters 
$X^*(\bbs)$ of $\bbs$ as an additive group, (see~\cite[Section 6.2.1]{Pe}, Section 3.1 and Lemma 11). For any $E\subseteq \Delta$, 
there is $\rho'_E\in X^*(\bbp_E)$ such that 
\begin{equation}\label{e:rho}
(\wedge^{\dim R_u(\bbp_E)}\Ad)(p) (u)=\rho_E'(p) u,
\end{equation}
for any $u\in\wedge^{\dim R_u(\bbp_E)}\Lie R_u(\bbp_E)$. It is well-known that 
the anti-canonical line-bundle of $\bbx$ corresponds to $\rho_E',$ see~\cite[Page 426]{FMT}. 

On the other hand, any $\chi=\sum_{\alpha\not\in E} n_\alpha \lambda_\alpha,$ where $n_\alpha$ are positive integers, is the highest weight of a unique irreducible representation $\eta_{\chi}:\bbg \rightarrow \GL(\bbv)$ which is strongly rational over $\bbq,$ \cite[Section 12]{BT}. So there is $0\neq v_0\in\bbv(\bbq)$ such that 
\[
\bbp_E=\{g\in \bbg| \eta_{\chi}(g)([v_0])=[v_0]\},
\]
where $[v_0]$ is the corresponding point in the projective space $\bbp(\bbv)$. 
Moreover, $\bbx$ is homeomorphic to the orbit $\eta_{\chi}(\bbg)([v_0])$. 

In the above setting, let $\ccal_E^+$ be the positive cone spanned by $\{\lambda_{\alpha}|\h \alpha\in E\}$
in the vector space $X^*\otimes\bbr,$ see Section~\ref{ss:IndexRelativeRootSystem} for the notation.

Let $H:\bbp(\bbv)(\bbq)\rightarrow \bbr^+$ be the ``usual'' height function, i.e. $H([v])=\|v\|$, where $v$ is a primitive integral vector and $\|\cdot\|$ is the Euclidean norm. The height function $H_{\chi}:\bbx\rightarrow \bbr^+$ associated with the metrized line-bundle $L_{\chi}$ is equal to $H(\eta_{\chi}(g)([v_0]))$. Let 
$
N_{\chi}(T):=\#\{x\in \bbx(\bbq)|\h H_{\chi}(x)\le T\}.
$
Franke, Manin and Tschinkel~\cite[Corollary 5]{FMT} gave the asymptotic behavior of 
$N_{\rho'_E},$ Batyrev and Tschinkel~\cite[Theorem 4.4]{BaT} did it for an arbitrary metrized line-bundle under the assumption that $\bbg$ is globally split, and L.~Peyre~\cite{Pe} gave the asymptotic for the quasi split case\footnote{It is worth mentioning that 
~\cite{Pe} gives a description of the constant $C$ in Theorem~\ref{t:Flag} in terms of Tamagawa numbers.
Moreover, in~\cite{FMT,BaT,Pe} the group $\bbg$ is defined over an arbitrary number field. In our formulation this can be replace by $\bbq$ using restriction of scalars.}. 
In the last section, we generalize their results to the case of an arbitrary almost simple $\bbq$-group $\bbg$ using an alternative approach.
\begin{thm}\label{t:Flag}
We have
\[
N_{\chi}(T)\sim C T^{\apzc_{\chi}} \log(T)^{\bpzc_{\chi}-1},
\]
where $\apzc_{\chi}:=\inf\{\apzc\in \bbr| \apzc\chi-\rho_E'\in \ccal_E^+\}$, $\bpzc_{\chi}$ is the codimension of the face of $\ccal_E^+$ which contains $\apzc_{\chi}\chi-\rho_E'$ and $C$ is a positive real number.
\end{thm}

\section{Preliminary and Notation}\label{s:notation}
\subsection{}\label{ss:IndexRelativeRootSystem}
Since the relative fundamental weights are not very well-known, for the 
convenience of the reader we review the following topics
\begin{enumerate}
	\item The index of $\bbg$.
	\item The relation between the index of $\bbg$ and its relative root system.
	\item Classification of strongly $\bbq$-rational irreducible representation.
\end{enumerate}
We refer the reader to \cite{BT,Spr} for the details and the proofs. 
Let $\bbg$, $\bbs$ and $\Delta$ be as in the introduction. That is: $\bbg$ is a simply connected almost simple $\bbq$-group which is $\bbq$-isotropic,
$\bbs$ is a maximal $\bbq$-split $\bbq$-torus and $\Delta$ is the associated simple roots relative to $\bbs$. 
Let $\bbt\supseteq\bbs$ be a maximal $\bbq$-torus 
in $\bbg$. Let $\TP\subseteq X^*=X^*(\bbt)$ (resp. $\TP^{\vee}\subseteq X_*=X_*(\bbt)$) be the (absolute) roots (resp. coroots) of $\bbg$ relative to 
$\bbt$. The absolute Galois group ${\rm Gal}(\overline{\bbq}/\bbq)$ and 
$(N_{\bbg}(\bbt)/\bbt)(\overline{\bbq})$ act linearly on $X^*$. If 
$\bbt$ is a $k$-split torus for a finite Galois extension $k$ of $\bbq$, then 
the absolute Galois group acts through $H={\rm Gal}(k/\bbq)$ and 
$\wt{W}=(N_{\bbg}(\bbt)/\bbt)(\overline{\bbq})=(N_{\bbg}(\bbt)/\bbt)(k)$. Let $(,)$ be a positive definite bilinear form on $E=X^*\otimes \bbr$ 
which is invariant under the action of $H\ltimes \wt{W}$. Using this 
bilinear form, we identify $X_*$ as the dual of $X^*$ in $E$. Let 
$j:\bbs \rightarrow \bbt$ be the injection map. It induces the surjection  
$j^*:X^*\rightarrow X^*(\bbs)$ and the injection  $j_*:X_*(\bbs)\rightarrow X_*$. Via the above maps, $X_*(\bbs)$ gets identified with 
\[
Y_*:=\{v\in X_*|\h H\cdot v=v\},
\]
$j^*$ gets identified with the orthogonal projection $\pi$ onto 
\[
V=\{v\in E|\h H\cdot v=v\},
\]
and $X^*(\bbs)$ gets identified with $Y^*:=\pi(X^*)$. The root system $\Phi$ 
of $\bbg$ relative to $\bbs$ is equal to $\pi(\TP)\setminus\{0\}$ and $\TD$ 
can be chosen such that $\Delta=\pi(\TD)\setminus\{0\}$. For any Galois 
automorphism $h$, there is a unique element $w_h$ in the (absolute) Weyl group such 
that $w_h(h(\TD))=\TD$. This induces a new linear action $\tau$ of the Galois 
group ${\rm Gal}(\overline{\bbq}/\bbq)$ on $X^*$. Let $\wt{\phi}\in\TP$ 
and $\pi(\wt{\phi})=\phi\in \Phi$. Then 
$\pi^{-1}(\phi)=\tau(H)(\wt{\phi})$. The absolute root system of the 
anisotropic kernel of $\bbg$ is given by $\Phi_0=\ker(\pi)\cap \TP$ and 
$\Delta_0=\TD\cap \Phi_0$ is a set of simple roots in $\Phi_0$. 

Under the above identifications, for any $\wt{\phi}\in \TP$ (resp. $\phi\in \Phi$), we have that $\wt{\phi}^{\vee}=2\wt{\phi}/(\wt{\phi},\wt{\phi})$ is a coroot (resp. 
$\phi^{\vee}=2\phi/(\phi,\phi)$ is a relative coroot). Let $\{\wt{\lambda}_{\wt{\alpha}}|\h {\wt{\alpha}}\in \TD\}$ be the (absolute) fundamental weights of $\bbg$, i.e. for any ${\wt{\alpha}},{\wt{\beta}}\in \TD$ we have
$(\wt{\lambda}_{\wt{\alpha}},{\wt{\beta}}^{\vee})=\delta_{{\wt{\alpha}}{\wt{\beta}}}$. For any $\alpha\in \Delta$, let 
\[
\lambda_{\alpha}:=\pi\left(\sum_{\wt{\alpha}\in \pi^{-1}(\alpha)} \wt{\lambda}_{{\wt{\alpha}}}\right).
\]
 These are called {\it the relative fundamental weights of} $\bbg$. A dominant 
weight $\wt{\chi}$ is the highest weight of an irreducible representation $\eta_{\wt{\chi}}$ which is strongly rational over $\bbq$ if and only if its restriction $\chi=\pi(\wt{\chi})$ to $\bbs$ is a positive integral linear combination of the relative fundamental weights. And the strongly rational representation $\eta_{\wt{\chi}}$ is uniquely determined by $\chi$. So it will be denoted by $\eta_{\chi}$. Note that 
\be\label{e:AlmostDuality}
(\lambda_{\alpha},\beta^{\vee})=\frac{\|\wt{\alpha}\|}{\|\alpha\|}\delta_{\alpha\beta},
\ee
where $\wt{\alpha}\in \pi^{-1}(\alpha)\cap \TD$. We also know that  $\lambda_{\alpha}$ is in the image of $X^*(\bbp_F)\rightarrow X^*(\bbs)$ if $\alpha\not\in F$, in which case its extension to $\bbp_F$ is also denoted by $\lambda_{\alpha}$.

\subsection{}\label{ss;G+}
In this note we will work with groups generated by unipotent subgroups both over $\bbr$ and $\bbq.$ We will use the following convention; if $\bbh$ is a group defined over $\bbq$, we let $\bbh^+$ denote the algebraic group generated $\bbq$-unipotent subgroups. If we let $H=\bbh(\bbr)$ then $H^+=\bbh(\bbr)^+$ denotes the group generated by $\bbr$-unipotent subgroups of $H.$ Let us note that $\bbh^+(\bbr)\subseteq\bbh(\bbr)^+,$ and the inclusion maybe strict. It is clear that $\bbh^+$ (resp. $H^+$) is a normal Zariski connected (resp. connected) subgroup of $\bbh$ (resp. $\bbh(\bbr)$). Hence, if $\bbh$ is $\bbq$-isotropic (resp. $\bbr$-isotropic) and $\bbq$-almost simple group (resp. $\bbr$-almost simple group), then $\bbh^+=\bbh^{\circ}$ (resp. $H^+=\bbh(\bbr)^{\circ}$).

\subsection{}\label{ss:Decompositions}
Let $G$ and $P$ be as in Section~\ref{s:statement}
and denote the Lie algebra of $G$ by $\mathfrak{g}.$  
Recall that $A=\bbs(\bbr)^\circ$ where $\bbs$ is a maximal $\bbq$-split $\bbq$-torus of $\bbg.$
Let $A'$ be a maximal $\bbr$-split torus of $G$ chosen so that
$A\subseteq A'\subseteq P$  
and let $\afr'$ be the Lie algebra of $A'.$ Let $\theta$ be a Cartan 
involution on $\gfr$ chosen such that if $\gfr=\mathfrak{k}+\mathfrak{q}$ is 
the corresponding Cartan decomposition of $\gfr$ then 
$\afr'\subseteq\mathfrak{q}.$ Let now $K$ be the analytic subgroup of $G$ with 
Lie algebra $\mathfrak{k}.$ Then $K$ is a maximal compact subgroup of $G$ and 
one has $G=KP$ (Iwasawa decomposition) and $G=KA'K$ (Cartan decomposition). In 
what follows $\|\cdot\|$ denotes a $K$-invariant Euclidean norm on $\gfr.$ We 
extend this to a norm on the exterior algebra $\wedge \mathfrak{g}$ of $\gfr,$ 
which we also denote by $\|\cdot\|$. With respect to the corresponding inner product, the $A'$-weight (as well as the $A$-weight)
spaces are orthogonal to each other.

\subsection{}\label{ss:Riemannian}
Let $B$ be the Killing form and define $g_0(X,Y)=-B(X,\theta(Y))$ on 
$\Lie(G).$ This defines a positive non-degenerate inner product on $\Lie(G)$ which induces a $G$-right invariant, $K$-bi  invariant Riemannian metric on 
$G.$ We will refer to this metric as {\em the Riemannian metric}. 

The negative of the Killing form turns $\afr'=\Lie(A')$ into a Euclidean space and moreover one can choose the Riemannian metric on $X$ such that $d(x_0,x_0\exp(a))=\|a\|$ for any $a\in \afr'$. On the other hand, the Lie algebra $\afr\subseteq \afr'$ of $A$ can be also viewed as $X^*(\bbs)\otimes \bbr$ (see Section \ref{ss:IndexRelativeRootSystem}).

For any $F\subseteq \Delta$, let $\rho_F$ be a vector in $\afr$ such that 
\be\label{e;rho-rho-notation}
\rho'_F(a)=e^{(\rho_F, \log a)},
\ee
for any $a\in A$, where $\rho'_F$ is defined by~\eqref{e:rho}.

Throughout the 
paper, for any (unimodular) Lie subgroup $H\subset G$ we let $\nu_H$
denotes the Haar measure on $H$ which is obtained from 
the Riemannian metric. In case $H=G,$ 
abusing the notation, we denote by $\nu_G$ the push 
forward of $\nu_G$ on quotients of $G$ as well. 

Let $\Gamma=\bbg(\bbz)$ (as in the introduction) and $\pi:G\rightarrow G/\Gamma$ be the natural projection. For any compactly 
supported continuous function $\psi$ function on $G/\Gamma$ and any measurable 
function $\phi$ on $G/\Gamma$ we let 
\[
\langle\phi,\psi\rangle=\int_{G/\Gamma}\phi(g\Gamma)\overline{\psi(g\Gamma)}d\nu_G.
\] 
For any $\mu\in \pcal(G/\Gamma)$ and a compactly supported continuous function $\psi\in C_c(G/\Gamma)$, we let $\mu(\psi):=\int_{G/\Gamma}\psi d\mu$.  

\subsection{}\label{ss:RationalPoints}
Recall that by a theorem of Borel and Harish-Chandra for any $\bbq$-parabolic 
subgroup $\bbp$ of $\bbg$ there exists a finite subset 
$\Xi\subseteq\bbg(\bbq)$ such that $\bbg(\bbq)=\bbp(\bbq)\cdot\Xi\cdot\Gamma.$ 
As we do not need the double coset representatives, by enlarging $\Xi,$ if 
necessary, we can and will assume that $\Xi$ is symmetric, i.e. $\Xi=\Xi^{-1}$. 
We also know that $(\bbg/\bbp)(\bbq)=\bbg(\bbq)/\bbp(\bbq)$ \cite[Lemma 2.6]{BT}.

\subsection{}\label{ss:MargulisDani}

We now recall the definition of certain functions $d_\alpha:G\rightarrow\bbr.$ 
These functions were considered by Dani and Margulis in~\cite{DM} in order to 
study the recurrence properties of unipotent flows on homogeneous spaces. Let 
$\Delta$ be as before. If $F=\Delta\setminus \{\alpha\}\subset\Delta$,  $\bbp_F$ is 
a  maximal $\bbq$-parabolic subgroup and it will be also denoted by $\bbp_\alpha$. 
Let $\bbu_\alpha=R_u(\bbp_\alpha)$ and let $\mathfrak{u}_\alpha$ denote the Lie algebra of 
$\bbu_\alpha.$ Let $\ell_\alpha=\dim\mathfrak{u}_\alpha$ and let 
$\vartheta_\alpha=\wedge^{\ell_\alpha}\mbox{Ad}$ denote the $\ell_\alpha$-th exterior power of the 
adjoint representation. Note that $\wedge^{\ell_\alpha}\ufr_\alpha$ defines a 
$\bbq$-rational one dimensional subspace of $\wedge^{\ell_\alpha}\gfr.$ Fix once and for all a unit vector 
$v_\alpha\in\wedge^{\ell_\alpha}\ufr_\alpha(\bbq).$ Note that if $g\in P_\alpha$ 
then $\vartheta_\alpha(g)v_\alpha=\det(\mbox{Ad}(g)|_{\ufr_\alpha})v_\alpha,$ and with the notation
as in~\eqref{e:rho} we have
$\rho_F'(g)=\det(\Ad(g)|_{\ufr_\alpha})$. 
Define $d_\alpha:G\rightarrow\bbr$ by 
\[
d_\alpha(g)=\|\vartheta_\alpha(g) v_\alpha\|\hspace{3mm}\mbox{for all}\h\h g\in G.
\]   

We also recall the following definition from \cite{EMS1}. For $d,N\in\bbn$ and $\Sigma>0$ let 
$\mathcal{P}(N,d,\Sigma)$ be the set of continuous maps $\Theta:\bbr^d\rightarrow G$ 
such that for all ${\bf c,a}\in \bbr^d$ with $\|c\|=1$ and any $X\in \mathfrak{g}$ the map 
\[
t\mapsto \Ad(\Theta(t{\bf c}+{\bf a}))(X)\in\mathfrak{g}
\]
has matrix coefficients of the form
$\sum_{j=1}^{N}\sum_{l=0}^{n-1}b_{jl}t^le^{\sigma_jt}$
where $b_{jl},\sigma_j\in\bbc$ and $|\sigma_j|\leq\Sigma.$ Let $\{X_1,\ldots,X_d\}\subset\mathfrak{g}$ and define 
\[
\mbox{$\Theta(\mathbf{t})=\exp(t_1X_1)\cdots\exp(t_dX_d)$ for all $\mathbf{t}=(t_1\dots,t_d)\in\bbr^d.$}
\]
Then there exists some $\Sigma>0$ and $N\in\bbn$ such that for any $g\in G,$ the map $\mathbf{t}\mapsto\Theta(\mathbf{t})g$ belongs to $\mathcal{P}(N,d,\Sigma).$

We will use the following two theorems in the sequel. These two theorems 
provide us with geometric and algebraic formulations of conditions which will 
guarantee a quantitative recurrence property of ``polynomial like" flows.  

\begin{thma*}~\label{shah}(Cf.~\cite[Theorem 3.4]{EMS1})
Let $G,$ $\Gamma$ be as in Theorem~\ref{trans}, and as before let $\pi: G\to G/\Gamma$ be the natural projection. 
Then, given a compact set $C\subseteq G/\Gamma,$ $\epsilon,\Sigma>0,$ $T>0$ and  $N,d\in\bbn$ there exists a compact set $L\subseteq G/\Gamma$ with the following property; for any $\Theta\in\mathcal{P}({N,d,\Sigma})$ 
and any open ball $\Omega$ in $\bbr^d$ with radius at most $T,$ one of the following holds

\begin{itemize}
\item[1.] $\frac{1}{m(\Omega)}m(\{\mathbf{t}\in\Omega\h|\h\pi(\Theta(\mathbf{t}))\in L\})>1-\epsilon,$ where $m$ denotes the Lebesgue measure on $\bbr^d.$\vspace{1mm}
\item[2.] $\pi(\Theta(\Omega))\cap C=\varnothing$
\end{itemize}
\end{thma*}

The second theorem gives an algebraic description as to why certain ``polynomial like" orbit stays outside a compact set for a long time. 

\begin{thmb*}~\label{ems}(Cf.~\cite[Theorem 3.5]{EMS1})
Let the notation be as in Theorem~\ref{trans}.
Suppose $d,N\in\bbn,$ $\beta,\Sigma>0,$ and $T>0$ are given. 
Then, there exists a compact set $C\subseteq G/\Gamma$ such that for any $\Theta\in\mathcal{P}({N,d,\Sigma}),$ and any open ball $\Omega\subseteq\bbr^d$ of radius at most $T,$ one  of the following conditions is satisfied
\begin{enumerate}
\item There exist $f\in \Xi$ (see \ref{ss:RationalPoints}), $\gamma\in\Gamma$ and $\alpha\in\Delta$ such that 
\[
\sup_{\mathbf{t}\in\Omega}\h d_\alpha(\Theta(\mathbf{t})f\gamma)<\beta,
\]
\item $\pi(\Theta(\Omega))\cap C\neq\varnothing$.
\end{enumerate}
\end{thmb*}
\section{Translates of horospherical measures}\label{sec;translates}
In this section Theorem~\ref{trans} is proved.  
\subsection{}\label{ss:PartOneEscape}
Let the notation be as in Sections~\ref{ss:IndexRelativeRootSystem}
and~\ref{ss:MargulisDani}. We start with a connection between $d_\alpha$ and the fundamental weights. 
By this connection, the first part of Theorem~\ref{trans} is proved. 
Let us recall that, for any $F\subseteq \Delta$, $\rho_F'\in X^*(\bbp_F)$ is a character such that for any $p\in \bbp_F(\bbr)$, 
\[
\rho'_F(p):=\det(\Ad(p)|_{\Lie(R_u(\bbp_F))(\bbr)}).
\]
\begin{lem}~\label{root}
For any $\alpha\in\Delta$, the restriction of  $\rho'_{\Delta\setminus\{\alpha\}}$ to $\bbs$ is equal to $k_\alpha\lambda_{\alpha}$ for some positive integer $k_\alpha$.
\end{lem}
\begin{proof} Let $F=\Delta\setminus\{\alpha\}$ and $\wt{F}=\TD\setminus \pi^{-1}(\alpha)$ (see Section \ref{ss:IndexRelativeRootSystem}). By the definition, the set of  positive (absolute) roots in $\ufr_\alpha$ is
\[
\wt{\Phi}_\alpha:=\{\wt{\phi}\in\wt{\Phi}|\h \exists\h \wt{\alpha}\in \TD\cap \pi^{-1}(\alpha): (\wt{\lambda}_{\wt{\alpha}},\wt{\phi})>0\}.
\]
And the restriction of $\rho_F'$ to $\bbs$ is equal to
$
\pi(\wt{\rho}_\alpha),
$
where $\wt{\rho}_\alpha=\sum_{\wt{\phi}\in\wt{\Phi}_\alpha} \wt{\phi}$.
For any $\wt{\beta}\in \wt{F}$, let 
\[
\sigma_{\wt{\beta}}(w):=w-(w,\wt{\beta}^{\vee})\wt{\beta}.
\]
Now we have that $\sigma_{\wt{\beta}}(\TP)=\TP$ and for any $\wt{\alpha}\in \TD\cap\pi^{-1}(\alpha)$
\[
(\wt{\lambda}_{\wt{\alpha}},\phi)=(\sigma_{\wt{\beta}}(\wt{\lambda}_{\wt{\alpha}}),\sigma_{\wt{\beta}}(\phi))=(\wt{\lambda}_{\wt{\alpha}},\sigma_{\wt{\beta}}(\phi)).
\]
Hence $\sigma_{\wt{\beta}}(\wt{\Phi}_\alpha)=\wt{\Phi}_\alpha$ and so $\sigma_{\wt{\beta}}(\wt{\rho}_\alpha)=\wt{\rho}_\alpha$ for any $\wt{\beta}\in \wt{F}$. Thus $\wt{\rho}_\alpha=\sum_{\wt{\alpha}\in \TD\cap \pi^{-1}(\alpha_\alpha)} k_{\wt{\alpha}}\wt{\lambda}_{\wt{\alpha}}$ where $k_{\wt{\alpha}}$ is a positive integer. On the other hand, for any Galois automorphism $h$, we have $\tau(h)(\wt{\Phi}_\alpha)=\wt{\Phi}_\alpha$. Hence for any $\wt{\alpha},\wt{\alpha}'\in \TD\cap \pi^{-1}(\alpha)$ we have $k_{\wt{\alpha}}=k_{\wt{\alpha}'}$. Therefore $\wt{\rho}_\alpha=k_\alpha\sum_{\wt{\alpha}\in \TD\cap \pi^{-1}(\alpha)} \wt{\lambda}_{\wt{\alpha}}$ and the restriction of $\rho_F'$ to $\bbs$ is equal to $k_\alpha \lambda_{\alpha}$.
\end{proof}
Recall the notation from Section~\ref{ss:MargulisDani}, in particular, for any $\alpha\in\Delta,$
$\mathfrak{u}_\alpha=\Lie (R_u(P_\alpha)),$ and $\ell_\alpha=\dim (\mathfrak{u}_\alpha).$

\begin{lem}~\label{out}
For any compact subset $C$ of $G$, if $\pi(aq)$ is in $\pi(C)$ for some $a\in A$ and $q\in Q_{E}$, then for any $\alpha\notin E$, 
\[
|\lambda_\alpha(a)|\ge \inf\{\|\vartheta_{\alpha}(c^{-1})\|^{-1/k_\alpha}|\h c\in C\},
\] 
where $\vartheta_{\alpha}=\wedge^{\ell_\alpha} \Ad:G\rightarrow \GL(\wedge^{\ell_\alpha} \gfr),$ and $k_\alpha$ is the constant in the Lemma~\ref{root}.
\end{lem}
\begin{proof}If $\pi(aq)\in\pi(C)$, then $aq=c\gamma$ for some $\gamma\in\Gamma$ and $c\in C.$ 
Let $v_\alpha$ be a non-zero vector with the shortest length in $\wedge^{\ell_\alpha} \mathfrak{u}_\alpha\cap\hspace{1mm} \wedge^{\ell_\alpha}\mathfrak{gl}_N(\bbz).$ By Lemma~\ref{root}, the definition of $\vartheta_{\alpha}$, and since $\alpha\notin E$ we have 
\[
\lambda_\alpha(a)^{k_\alpha}=\vartheta_{\alpha}(aq) v_\alpha=\vartheta_{\alpha}(c\gamma) v_\alpha. 
\]
Because of the arithmetic structure of $\Gamma$, $\|\vartheta_{\alpha}(c\gamma)(v_\alpha)\|\ge \|\vartheta_{\alpha}(c^{-1})\|^{-1}$. Altogether $|\lambda_\alpha(a)|^{k_\alpha}\ge\|\vartheta_{\alpha}(c^{-1})\|^{-1}$. 
\end{proof}

\begin{proof}[Proof of part (i) of Theorem~\ref{trans}] It is enough to show that, under the assumption on $\{a_n\}$, for any compact subset $\bar{C}$ of $G/\Gamma,$ $(a_n \mu_{Q_E})(\bar{C})$ tends to zero as $n$ goes to infinity. One can consider $\bar{C}$ as $\pi(C)$ where $C$ is a compact subset of $G$. Since $\mbox{supp}(a_n \mu_{Q_E})=a_n Q_E$, to complete the proof it is enough to note that $\pi(a_n Q_E)\cap\pi(C)$ is empty for large enough $n$ which is a consequence of our hypothesis and Lemma~\ref{out}.
\end{proof}
\subsection{}\label{ss:SecondPartConvergence}
In order to prove the second part of Theorem~\ref{trans}, first we prove the special case where $\varnothing=E\subseteq F\subseteq \Delta.$ For simplicity in notation we let $Q=Q_{\varnothing}$.

\begin{thm}\label{thm;U-trans}
Let $F\subseteq \Delta$. Let $\{a_n\}$ be a sequence of elements in $A$ such that
\begin{enumerate}
	\item for any $\alpha\in \Delta\setminus F$ and $n\in\bbn$, $\lambda_{\alpha}(a_n)=1$,
	\item for any $\alpha\in F$, $\lim_{n\rightarrow\infty}\lambda_{\alpha}(a_n)=\infty$.
\end{enumerate}
Then $a_n \mu_Q$ converges to $\mu_{Q_F}$. In particular, when $F=\Delta$, it 
converges to the Haar measure on $G/\Gamma$.
\end{thm}

The proof of Theorem~\ref{thm;U-trans} is carried out in two steps. First in Section~\ref{sss:NoEscape} using the recurrence properties of certain flows, we show that under the given conditions $a_n \mu_{Q}$ does not escape to infinity (see Proposition~\ref{p:NoEscape}). Then we will complete the proof in Section~\ref{sss:Convergence}.

\subsubsection{}\label{sss:NoEscape}
The goal of this section is to use the recurrence properties of certain flows to prove Proposition~\ref{p:NoEscape}.
\begin{prop}~\label{p:NoEscape}
Let $\wt{X}$ be the one-point compactification of $X=G/\Gamma$ and $\{a_n\}$ be as in Theorem~\ref{thm;U-trans}. If $\mu$ is the limit of a subsequence of $\{a_n\mu_Q\}$ in $\pcal(\wt{X})$, then $\mu(\{\infty\})=0$. 
\end{prop}
\begin{remark}\label{r:NoEscape}
Proposition~\ref{p:NoEscape} is equivalent to saying that for any $\vare>0$ there is a compact subset of $C'\subseteq G,$ such that $(a_n \mu_Q)(\pi(C'))>1-\vare$. 
\end{remark}

The proof is similar to the proof of~\cite[Theorem 1.1]{EMS1}. The major difference is that our subgroup $Q$ is not reductive. What pays the price here is the stated conditions on the translating sequence $\{a_n\}.$ Let $\{X_1,\dots, X_d\}$ be a basis for the Lie algebra of $Q$ and define 
\[
\mbox{${\Theta}:\bbr^k\rightarrow G$\hspace{2mm} by \hspace{2mm}${\Theta}(\mathbf{t})=\exp(t_1X_1)\cdots\exp(t_dX_d).$}
\] 
Then there exists $N,\Sigma$ such that for any $g,g'\in G$ the map ${\bf t}\mapsto g{\Theta}({\bf t})g'$ is in $\mathcal{P}(N,d,\Sigma).$

Let $B$ be a small ball in $\bbr^d$ such that $\exp$ is a diffeomorphism into $G$ and~\cite[Proposition 3.6]{EMS2} holds on $B.$\footnote{This choice is to guarantee the ``polynomial like'' behavior of ${\Theta}.$} Note that push forward of the Lebesgue measure on $\Theta(B)$ and the restriction of $\nu_Q$ on $\Theta(B)$ are absolutely continuous with respect to each other. Moreover $\Theta$ is a diffeomorphism whose derivative at the origin is the identity matrix. Hence taking $B$ small enough we may and will assume that the Radon-Nikodym derivative of the former with respect to the latter is bounded above and below by $2$ and $1/2.$

For any $\alpha\in\Delta$, let $W_\alpha$ be the sum of the highest weight spaces in the representation $V_\alpha=\wedge^{\ell_\alpha}\mathfrak{g}$, where $\ell_\alpha:=\dim \ufr_\alpha$.  
Recall that $Q$ contains a maximal unipotent subgroup of $G.$ Thus $q{\Theta}(B)q^{-1},$ for any $q\in Q,$ contains a neighborhood of the identity in this maximal unipotent subgroup. This, thanks to~\cite[Lemma 1.4]{Sh}, implies that for any $q\in Q$ there is a constant $e,$ depending on $q,$ such that: 
\begin{equation}\label{e;shah}
\|v\|\le e\cdot \sup_{{\bf t}\in \ocal} \|\mbox{pr}_{W_\alpha}(\vartheta_\alpha(q{\Theta}({\bf t})) v)\|. 
\end{equation}
We conclude from this the following
\begin{lem}~\label{d_ilowerbound}
Let $\ocal\subseteq B$ be any ball and $\Xi\subseteq\bbg(\bbq)$ and $\{q_1,\dots,q_m\}\subset Q$ be finite sets. Then there is some $\theta>0$ such that; if $|\lambda_{\alpha}(a)|\ge 1$ for all $\alpha,$ then 
\[
\sup_{{\bf t}\in \ocal} d_\alpha(aq_i{\Theta}({\bf t})f\gamma)\ge \theta,
\] 
for all $1\leq i\leq m,$ $\alpha\in\Delta,$ $\gamma\in \Gamma,$ and $f\in \Xi$. 
\end{lem}
\begin{proof}
By the definition of $d_\alpha$, we have:
\[
d_\alpha(aq_i{\Theta}({\bf t})f\gamma)=\|\vartheta_\alpha(aq_i{\Theta}({\bf t})f \gamma)v_\alpha\|.
\] 
It is clear that there is a constant $e'$ depending on the set $\Xi$ (and independent of $\gamma$) such that $\|\vartheta_\alpha(f \gamma)v_\alpha\|\ge e'$. On the other hand, because the weight spaces are assumed to be orthogonal to one another, we have 
\[
\|\vartheta_\alpha(a{\Theta}({\bf t})f \gamma)v_\alpha\|\ge\|\vartheta_\alpha(a)(\mbox{pr}_{W_\alpha}(\vartheta_\alpha(q{\Theta}({\bf t})f \gamma) v_\alpha))\|.
\]  
Since any highest weight is a product of fundamental weights and the fundamental weights at $a$ are at least 1, by the definition of $W_\alpha,$ we have 
\[
\|\vartheta_\alpha(a)({\rm pr}_{W_\alpha}(\vartheta_\alpha({\Theta}({\bf t})f \gamma) v_\alpha))\|\geq\|{\rm pr}_{W_\alpha}(\vartheta_\alpha({\Theta}({\bf t})f \gamma) v_\alpha)\|.
\]  
So one can finish the argument using~\eqref{e;shah}.
\end{proof}

\begin{proof}[Proof of Proposition~\ref{p:NoEscape}]
Since $Q=Q_\varnothing$ by the Godement compactness criterion $Q\cap\Gamma$ is a uniform lattice in $Q.$ Let $\vare>0$ be given. 
Let $\mathcal{O}\subset B$ be a small open ball such that  for all $x\in Q\Gamma/\Gamma,$ the map $q\mapsto qx$ is injective on $\Theta(\ocal).$ 
Let $\mathcal{Q}\subset Q$ be a compact subset such that $Q\Gamma/\Gamma=\mathcal{Q}\Gamma/\Gamma.$ Let $\{q_1,\dots,q_m\}\subset Q$ be so that $\mathcal{Q}=\cup_{i} q_i{\Theta}(\ocal).$ Further assume that the multiplicity of this covering is at most $\kappa.$ Using Lemma~\ref{d_ilowerbound} and Theorem B for each $1\leq i\leq m,$ there exits some compact set $C_i\subset G$ such that $\pi(a_nq_i{\Theta}(\ocal))\cap \pi(C_i)$ is nonempty for all $n.$ Let $C'_i$ be the compact set given by Theorem A applied to $C_i$ and $B,$ that is; we have 
\be\label{e;high-dens-ret}
m(\{{\bf t}\in \ocal | \pi(a_nq_i{\Theta}({\bf t}))\in \pi(C'_i) \})\ge (1-\vare) m(\ocal),
\ee
where $m$ is the Lebesgue measure on $\bbr^d$. Let $C'=\cup_i C'_i,$ then~\eqref{e;high-dens-ret} holds for any $1\leq i\leq m$ and $C'$ in place of $C'_i.$ This, in view of the choice of $B,$ implies that 
\[
\nu_Q(\{\Theta({\bf t})| \pi(a_nq_i{\Theta}({\bf t}))\in \pi(C')\})\ge (1-4\vare)\nu_Q(\Theta(\ocal)).
\]
Now the definition of $\mu_Q$ and the choice of $\mathcal{O}$ and $q_i$'s imply that  
\[
(a_n\mu_Q)(\pi(C'))\ge 1-4\kappa\vare,
\] 
and this finishes the proof by Remark~\ref{r:NoEscape}.
\end{proof}
\subsubsection{}\label{sss:Convergence}
In this Section first we complete the proof of Theorem~\ref{thm;U-trans} by proving Proposition~\ref{p:FindTheLimitUTrans} and then we prove the second part of Theorem~\ref{trans}.

\begin{prop}\label{p:FindTheLimitUTrans}
Let $\{a_n\}$ be as in Theorem~\ref{thm;U-trans}. Assume $\{a_n\mu_Q\}$ converges to $\mu\in\pcal(G/\Gamma)$. Then $\mu=\mu_{Q_F}$.
\end{prop}

To prove Proposition~\ref{p:FindTheLimitUTrans}, the main results of Mozes and Shah~\cite{MS}, specially Theorem 1.1, will be used. Before going into the proof of Proposition~\ref{p:FindTheLimitUTrans}, let us briefly recall some of the properties of parabolic $\bbq$-subgroups.
 
\begin{lem}\label{l:CharParabolic}
Let the notation be as in Section~\ref{ss:IndexRelativeRootSystem} (recall that $\bbg$ is simply-connected). For any $F\subseteq \Delta$ 
\[
{\rm Im}(X^*(\bbp_F)_{\bbq}\rightarrow X^*(\bbs))=
\bigoplus_{\alpha\in \Delta\setminus F} \bbz \lambda_{\alpha},
\]
 where $X^*(\bbp_F)_{\bbq}$ is the set of characters which are defined over $\bbq$.
\end{lem}
\begin{proof}
Since $\bbg$ is simply-connected, it is well-known (see Section~\ref{ss:IndexRelativeRootSystem} and~\cite[Lemma 8.1.4]{Spr}) that
\be\label{e:CharParaAbsolute}
{\rm Im}(X^*(\bbp_F)\rightarrow X^*(\bbt))=\bigoplus_{\wt{\alpha}\in \TD\setminus (\pi^{-1}(F)\cup \Delta_0)} \bbz \wt{\lambda}_{\wt{\alpha}}.
\ee
Let $\lambda\in {\rm Im}(X^*(\bbp_F)_{\bbq}\rightarrow X^*(\bbs))$. So by Equation (\ref{e:CharParaAbsolute}), 
\[
\lambda=\sum_{\wt{\alpha}\in \TD\setminus (\pi^{-1}(F)\cup \Delta_0)} c_{\wt{\alpha}} \wt{\lambda}_{\wt{\alpha}}.
\]
On the other hand, since $\lambda$ is defined over $\bbq$, for any $h\in {\rm Gal}(\overline{\bbq}/\bbq)$, we have $\tau(h)(\lambda)=\lambda$. Hence
\be\label{e:Coef}
c_{\wt{\alpha}}=(\wt{\alpha}^{\vee},\lambda)=
(\tau(h)(\wt{\alpha}^{\vee}),\lambda)=(\tau(h)(\wt{\alpha})^{\vee},\lambda)=c_{\tau(h)(\wt{\alpha})}.
\ee
Thus for any $\wt{\alpha}\in \TD\setminus (\pi^{-1}(F)\cup \Delta_0)$, $c_{\wt{\alpha}}$ only depends on $\alpha:=\pi(\wt{\alpha})$ and so we will denote it by $c_{\alpha}$. Therefore
\[
\lambda=\pi(\lambda)=\sum_{\alpha\in \Delta\setminus F} c_{\alpha} \lambda_{\alpha},
\]
which finishes the proof.
\end{proof}
By Lemma~\ref{l:CharParabolic} we have that the restriction of $\rho_F'$ to $\bbs$ (which will be denoted by $\rho_F'$ again) is in $\sum_{\alpha\in \Delta\setminus F} \bbz \lambda_{\alpha}$. In the next lemma (for the future use) we recall that $\rho_F'$ is in fact in the positive cone generated by $\{\lambda_{\alpha}\}_{\alpha\in\Delta\setminus F}$. Notice that it is a  generalization of Lemma~\ref{root}. 

\begin{lem}\label{l:RhoPositiveCone}
Let $\bbg, \bbs, \Delta$ be as in Section~\ref{ss:IndexRelativeRootSystem}. As before for any $F\subseteq \Delta$, let $\rho_F'\in X^*(\bbs)$ be a character induced by
\[
a \mapsto \det(\Ad(a)|_{\Lie(R_u(\bbp_F))}).
\] 
Then $\rho_F'\in \sum_{\alpha\in\Delta\setminus F}\bbz^+ \lambda_{\alpha}.$
\end{lem}
\begin{proof}
First let us assume that $\bbg$ splits over $\bbq$ i.e. $\bbs$ is a maximal torus in $\bbg$. Let $\Psi_F$ be the subset of the positive roots $\Phi^+$ which consists of integer combination of elements of $F$. Let $\Phi_F:=\Phi^+\setminus \Psi_F$. Then (since $\bbg$ is $\bbq$-split) we have
\be\label{e:RhoSumRoots}
\rho_F'=\sum_{\phi\in\Phi_F}\phi.
\ee
Since any root $\phi$ can be written as either a positive or a negative linear combination of simple roots, $\sigma_{\alpha}(\Phi^+\setminus \{\alpha\})=\Phi^+\setminus \{\alpha\}$ for any $\alpha\in \Delta$. By a similar logic, one can show that for $\alpha\in \Delta\setminus F$ we have
\be\label{e:UnstableRoots}
\{\phi\in \Psi_F|\h \sigma_{\alpha}(\phi)\not \in \Psi_F\}=\{\phi\in \Psi_F|\h (\phi,\alpha^{\vee})\neq 0\}.
\ee
Let us denote the set in Equation (\ref{e:UnstableRoots}) by $\Psi_{F,\alpha}$. Notice that as $\sigma_{\alpha}(\phi)=\phi-(\phi,\alpha^{\vee})\alpha\in \Phi^+$ for any $\phi\in \Psi_{F,\alpha}$, we have that
\be\label{e:ObtuseRoots}
\forall\h \alpha\in \Delta\setminus F,\h \forall\h \phi\in \Psi_{F,\alpha},\h\h\h\h\h (\phi,\alpha^{\vee})<0. 
\ee
For any $\alpha\in \Delta\setminus F$, let $\Phi_{F,\alpha}:=\Phi_F\setminus (\sigma_{\alpha}(\Psi_{F,\alpha})\cup\{\alpha\})$. Then by the above discussion we have
$\sigma_{\alpha}(\Phi_{F,\alpha})=\Phi_{F,\alpha}$. Therefore 
\be\label{e:AlphaComponentRho}
\sigma_{\alpha}(\rho_F')-\rho_F'=-2 \alpha+\sum_{\phi\in \Psi_{F,\alpha}}(\phi,\alpha^{\vee}) \alpha,
\ee 
which implies that $(\rho_F',\alpha^{\vee})=2-\sum_{\phi\in \Psi_{F,\alpha}}(\phi,\alpha^{\vee})$. Now, by (\ref{e:ObtuseRoots}), (\ref{e:AlmostDuality}), and Lemma~\ref{l:CharParabolic}, we have that $\rho_F'\in \bbz^+ \lambda_{\alpha}$.

Now we use the basics of the index and the relative root system of $\bbg$ as described in Section~\ref{ss:IndexRelativeRootSystem} to get the general case. As in Section~\ref{ss:IndexRelativeRootSystem},
the absolute roots, the coroots, etc will be denoted by an extra tilde and $\pi$ is the orthogonal projection
from $X^*(\bbt)$ to $X^*(\bbs)$. Let $\Delta_0=\TD\cap \ker(\pi)$ and $\wt{F}=\pi^{-1}(F)\cap \TD$. 
By the definition of $\rho_F'$ we have
\be\label{e:RelativeRho}
\rho_F'=\pi\left(\sum_{\wt{\phi}\in\TP_{\wt{F}\cup \Delta_0}} \wt{\phi}\right),
\ee
where $\TP_{\wt{F}\cup \Delta_0}\subseteq \TP^+$ consists of those (absolute) roots that are not a linear combination of elements of $\wt{F}\cup \Delta_0$. 

Now from the split case we know that there are $\wt{c}_{\wt{\alpha}}\in \bbz^+$ such that
\be\label{e:AbsoluteRho}
 \sum_{\wt{\phi}\in\TP_{\wt{F}\cup \Delta_0}} \wt{\phi}=\sum_{\wt{\alpha}\in\TD\setminus (\wt{F}\cup\Delta_0)} \wt{c}_{\wt{\alpha}}\wt{\lambda}_{\wt{\alpha}}.
\ee
Now by a similar argument as in the proof of Lemma~\ref{l:CharParabolic} using the fact that $\rho_F'$ is invariant under the Galois action one can deduce that $\wt{c}_{\wt{\alpha}}$ depends only on $\alpha:=\pi(\wt{\alpha})$ and it will be denoted by $c_{\alpha}$. And so
\[
\rho_F'=\pi\left(\sum_{\wt{\phi}\in\TP_{\wt{F}\cup \Delta_0}} \wt{\phi}\right)=
\pi\left(\sum_{\wt{\alpha}\in\TD\setminus (\wt{F}\cup\Delta_0)} \wt{c}_{\wt{\alpha}}\wt{\lambda}_{\wt{\alpha}}\right)=
\pi\left(\sum_{\alpha\in\Delta\setminus F}c_{\alpha}\sum_{\wt{\alpha}\in \pi^{-1}(\alpha)\cap \TD} \wt{\lambda}_{\wt{\alpha}}\right)=
\sum_{\alpha\in\Delta\setminus F}c_{\alpha}\lambda_{\alpha}.
\]
\end{proof}

\begin{lem}\label{l:Reductive}
Let $\bbh$ be a Zariski-connected reductive $\bbq$-group and $\bbs$ a maximal $\bbq$-split $\bbq$-torus in $\bbh$.  
If $X^*(\bbh)_{\bbq}=\{0\}$, then $\bbs\subseteq \bbh^+$. 
\end{lem}
\begin{proof}
It is well-known that
\begin{enumerate}
	\item $\bbh=Z(\bbh)\cdot [\bbh,\bbh]$, where $Z(\bbh)$ is the center of $\bbh$.
	\item $[\bbh,\bbh]$ is a semisimple $\bbq$-group.
	\item $[\bbh,\bbh]$ can be written as almost direct product $\bbh_1\bbh_2 \cdots\bbh_m$ of its almost $\bbq$-simple $\bbq$-groups $\bbh_i$.
	\item Any unipotent subgroup of $\bbh$ is a subset of $[\bbh,\bbh]$. In particular, $\bbh^+\subseteq [\bbh,\bbh]$. 
\end{enumerate}
On the other hand, by \cite[Paragraph 12.1]{Bor}, since $X^*(\bbh)_{\bbq}=\{0\}$, $\bbs\subseteq [\bbh,\bbh]$. So without loss of generality we can and will assume that $\bbh$ is the almost product of its almost $\bbq$-simple factors $\bbh_i$. Let us also assume that the first $m'$ factors are the $\bbq$-isotropic ones. Thus by \cite[Corollaire 6.9]{BoT2} 
\be\label{e:H+}
\bbh^+=\bbh_1\bbh_2\cdots\bbh_{m'}.
\ee
On the other hand, it is easy to see that $\bbs$ is almost product of the Zariski-connected component of $\bbs\cap \bbh_i$. Hence 
\be\label{e:S}
\bbs\subseteq \bbh_1\bbh_2\cdots\bbh_{m'}.
\ee
Equations (\ref{e:H+}) and (\ref{e:S}) give us the claim.
\end{proof}
\begin{lem}\label{l:SandP1}
Let the notation be as in Section~\ref{ss:IndexRelativeRootSystem}. For any $F\subseteq \Delta$, let $\bbp_F^{(1)}=\bigcap_{\lambda\in X^*(\bbp_F)_{\bbq}}\ker(\lambda)$. Then 
\[
\bbs_F^{(1)}:= (\bbs\cap \bbp_F^{(1)})^{\circ}\subseteq \bbp_F^+.
\]
\end{lem}
\begin{proof}
Let $\bbh:=\bbp_F^{(1)}/R_u(\bbp_F)$. It is clear that $\bbh$ is a $\bbq$-subgroup of $\bbh$, $\bbh^+=\bbp_F^+/R_u(\bbp_F)$, and $\bbs_F^{(1)}$ injects into $\bbh$. Now using Lemma~\ref{l:Reductive} one can finish the proof.   
\end{proof}

\begin{lem}~\label{ourversion} In the setting of Theorem~\ref{thm;U-trans}, 
assume $\{a_n\mu_Q\}$ converges to $\mu\in\pcal(G/\Gamma).$
Then, there are a Zariski-connected $\bbq$-subgroup $\bbh$ of $\bbg$ and $g\in G$ with the following properties.
\begin{enumerate}
	\item $\pi(H)$ is a closed subset of $\pi(G)$ and $\mu=g\mu_H$, where $H$ is the connected component of $\bbh(\bbr)$ (recall that $\pi:G\to G/\Gamma$ is the quotient map).
	\item $\Gamma\cap H$ is a lattice in $H$ and it is Zariski-dense in $\bbh$. 
	\item There is a sequence $\{g_n\}\subseteq G$ such that
			\begin{enumerate}
				\item $g_n\rightarrow e$ as $n$ goes to infinity.
				\item For any $n,$ $\pi(a_n Q)\subseteq \pi(g_ngH)$.
				\item For any $n$, $g_n^{-1}Qg_n\subseteq gHg^{-1}\subseteq Q_F$.
			\end{enumerate}	
		\item For some $E\subseteq F$, $\bbp^+_E\subseteq g\bbh g^{-1} \subseteq \bbp_E^{(1)}.$  	
\end{enumerate}
\end{lem}
\begin{remark}
When $\bbg$ is a {\em $\bbq$-split} simply connected $\bbq$-almost simple group, then $\bbp^+_E=\bbp^{(1)}_E$ for any $E\subseteq \Delta$. In particular, in this case, $g\bbh g^{-1}=\bbp^+_E$ and $Q_E=\bbp^+_E(\bbr)$.
\end{remark}
\begin{proof}
Recall that $\bbp_\varnothing(\bbr)^+\subseteq Q$ and it acts topologically transitively on $\pi(Q).$ By the definition $\bbp_\varnothing(\bbr)^+$ is generated by one parameter $\bbr$-unipotent subgroups. Let $\{u^{(i)}(t)\}$ be a finite set of one-parameter unipotent subgroups of $Q$, such that $\mu_Q$ is $u^{(i)}(t)$-ergodic for any $i$, and the group generated by them is $\bbp_\varnothing(\bbr)^+$. Let $Y$ be the subset of $\pi(Q)$ which consists of $u^{(i)}(t)$-generic points for all $i$. Hence $Y$ is a $\mu_Q$-co-null set and for any $y \in Y$ 
\[
\lim_{T\rightarrow\infty}\frac{1}{T}\int^T_0 f(u^{(i)}(t)y) dt=
\int_{\pi(Q)} f \h d\mu_Q,
\] 
for any (bounded) continuous function $f$ on $\pi(Q)$.

For any $n$ and $i$, $\mu_n=a_n \mu_Q$ is $u^{(i)}_n(t)=a_n u^{(i)}(t) a_n^{-1}$-invariant and ergodic. Therefore by \cite[Corollary 1.1]{MS} and \cite{D1, D2}, $\mu$ is algebraic, i.e. there is a closed subgroup $H$ of $G$ such 
that $\pi(H)$ is a closed subset of $G/\Gamma$, and $g\in G$ such that 
$\mu=g\mu_H$. This shows half of part 1.

 By \cite[Theorem 1.1 and Proposition 2.1]{MS}, we have 
\begin{enumerate}
	\item $\mu$ is $gH^{\circ}g^{-1}$-ergodic where $H^{\circ}$ is the connected component of $H$.
	\item $\Gamma\cap H$ is a lattice in $H$.
	\item $\Gamma\cap H$ is Zariski-dense in $H$.
\end{enumerate}
Hence $H$ is connected and there is a Zariski-connected $\bbq$-subgroup $\bbh$ of $\bbg$ such that $\bbh(\bbr)^{\circ}=H$, which completes the proof of part 1 and also gives part 2. 

Let $\{g_n\}$ be a sequence of elements of $G$ such that 
\begin{itemize}
\item[i)] $g_n\rightarrow e,$ as $n\rightarrow \infty$.
\item[ii)]$\pi(g_n g)$ is a generic point for all the unipotent flows $u^{(i)}_n(t)$ with respect to $\mu_n$.
\end{itemize}
Then by \cite[Theorem 1.1, Part 1 and 2.]{MS} and using the fact that $u^{(i)}(t)$ generate $Q$ one can finish the proof of part 3, except for the last inclusion in (c). 

By \cite[Theorem 1.1, part 1.]{MS} we also have 
\be\label{e:MeasureStabilizer}
gHg^{-1}=\{g'\in G|\h g'\mu =\mu\}.
\ee
Since $\mu_n$'s are $Q$-invariant so is $\mu$, hence by Equation~(\ref{e:MeasureStabilizer}), $Q\subseteq gHg^{-1}$. In particular we have  $R_u(\bbp_\varnothing)\subseteq g\bbh g^{-1}$ and by~\cite[Proposition 8.6]{BT} there is $E\subseteq \Delta$ such that 
\begin{enumerate}
	\item $\bbp^+_E\subseteq g\bbh g^{-1}\subseteq \bbp_E,$
	\item $R_u(g\bbh g^{-1})=R_u(\bbp_E)$.
\end{enumerate}
It is worth mentioning that we have $\bbp_E^+\subseteq g\bbh g^{-1}$ because of~\cite[Proposition 8.6]{BT} and the fact that $\bbp_E^+$ is the smallest normal subgroup of
$\bbp_E$ which contains $R_u(\bbp_\varnothing).$

On the other hand, by Lemma~\ref{l:CharParabolic} and our assumption on $\{a_n\}$, we have that $a_n\in \bbp_F^{(1)}$ for any $n$. Hence by Lemma~\ref{l:SandP1}, $a_n\in \bbs_F^{(1)}\subseteq \bbp_F^+$ for large enough $n$. Thus $a_nQ\subseteq Q_F$ for large enough $n,$ by the definition of $Q_F$. Therefore $a_n\mu_Q$ can be also viewed as elements of $\pcal(\pi(Q_F))$. Hence by \cite[Corollary 1.1]{MS} $\mu$ is also in $\pcal(\pi(Q_F))$. Hence we get $gHg^{-1}\subset Q_F$ which completes part 3. We also get $\bbp^+_E \subseteq \bbp_F$, which shows that $E\subseteq F$. 

Let $\bbl_E$ be the standard Levi factor of $\bbp_E$. It is clear that the intersection $Z_{\bbp_E}(R_u(g\bbh g^{-1}))\cap \bbl_E$ of the centralizer of $R_u(g\bbh g^{-1})=R_u(\bbp_E)$ in $\bbp_E$ and $\bbl_E$ is in the center of $\bbg$. Hence by a corollary of Auslander-Wang's theorem~\cite[Corollary 8.28]{Rag} $g\Gamma g^{-1}\cap R(g\bbh g^{-1})(\bbr)$ is a lattice in $R(g\bbh g^{-1})(\bbr)$, where $R(g \bbh g^{-1})$ is the radical of $g\bbh g^{-1}.$ 
Therefore by a theorem of Borel and Harish-Chandra we have $X^*(g \bbh g^{-1})_{\bbq}=\{1\}$ and so $g \bbh g^{-1}\subseteq \bbp_E^{(1)}$, which completes the proof of part 4.    
\end{proof}
 
\begin{lem}~\label{parabolic}
Let $\bbu=R_u(\bbp)$ be the unipotent radical of the standard minimal $\bbq$-parabolic and let $F\subseteq \Delta$. Then 
\[
T(\bbu,\bbp_F):=\{q\in \bbg|\h q^{-1}\bbu q\subseteq \bbp_F\}=\bbp_F.
\]
\end{lem}
\begin{proof} 
After using Bruhat decomposition, without loss of generality, we can assume that $q\in N_{\bbg}(\bbs)$. Therefore $q^{-1}\bbp q=q^{-1}Z_{\bbg}(\bbs)\bbu q\subseteq \bbp_F$. So $q\in T(\bbp,\bbp_F)$. However the later is equal to $\bbp_F$ (See~\cite[Proposition 4.4]{BT}).
\end{proof}
\begin{proof}[Proof of Proposition~\ref{p:FindTheLimitUTrans}]
Let $E\subseteq \Delta$ be as in Lemma~\ref{ourversion}. 
We first prove that $E=F.$ Assume to the contrary that $\alpha\in F\setminus E$.
 
Since $g\bbh g^{-1}\subseteq \bbp_E$, by Lemma~\ref{ourversion}, part 3(c), 
$g_n\in T(\bbu,\bbp_E)$, for any $n$. Hence by Lemma~\ref{parabolic}, $g_n\in 
\bbp_E$. By Lemma~\ref{ourversion}, part 3(b), for any $n$, there are 
$h_n\in H$ and $\gamma_n\in\Gamma$ such that 
\be\label{good!}
a_n=g_n g h_n \gamma_n=g_n (g h_n g^{-1}) (g\gamma_n).
\ee 
Since $a_n, g_n, g h_n g^{-1}$'s are in $P_E$, so are $g \gamma_n$'s and $\gamma_1^{-1}\gamma_n$. 

We note that $\lambda_{\alpha}(g h_n g^{-1})=1$ since $gHg^{-1}\subseteq \bbp^{(1)}_E$. We also know that $|\lambda_{\alpha}(\gamma_1^{-1}\gamma_n)|=1$ as $\lambda_{\alpha}$ is a $\bbq$-character and $\Gamma$ preserves the $\bbz$-structure. Thus
\[
|\lambda_{\alpha}(g\gamma_n)|=|\lambda_{\alpha}(g\gamma_1)|\cdot|\lambda_{\alpha}(\gamma^{-1}_1\gamma_n)|=|\lambda_{\alpha}(g\gamma_1)|.
\] 
Altogether we have
\[
|\lambda_{\alpha}(a_n)|=|\lambda_{\alpha}(g_n)|\cdot|\lambda_{\alpha}(g\gamma_1)|,
\] 
which is a contradiction since the left hand side tends to infinity and the right hand side remains bounded as $n$ goes to infinity.

Using (c) part 3 of Lemma~\ref{ourversion}, we have $Q\subseteq gHg^{-1}\subseteq Q_E.$ Hence thanks to part 4 of Lemma~\ref{ourversion} the proof will be complete if we show $\bbp_E(\bbr)^+\subseteq \langle\bbp_E^+(\bbr),Q\rangle$. Let 
$\bbp_E=\bbl_E R_u(\bbp_E),$ where $\bbl_E$ is the standard Levi factor. Let $\bbl_E=Z(\bbl_E)[\bbl_E,\bbl_E]$ and $\bbm=[\bbl_E,\bbl_E].$ 
Then, $\bbm$ is a semisimple $\bbq$-group and can be written as almost direct product
$\bbm=\bbh_1\cdots\bbh_k$ of $\bbq$-almost simple groups.
In order to complete the proof we need to show $\bbm(\bbr)^+\subseteq\langle\bbp_E^+(\bbr),Q\rangle.$

Note that if $\bbh$ is a $\bbq$-isotropic $\bbq$-almost simple group, then $\bbh(\bbr)^+=\bbh^+(\bbr).$ Hence there is nothing to prove if all $\bbh_i$'s are $\bbq$-isotropic. Thus let us assume $\bbh_i$ is $\bbq$-anisotropic for some $i.$
Therefore, $R_u(\bbp_\varnothing)\cap\bbh_i$ is trivial and so $\bbh_i$ centralizes 
$R_u(\bbp_\varnothing)\cap\bbl_E.$ This implies that $\bbh_i$ normalizes
$R_u(\bbp_\varnothing)=(R_u(\bbp_\varnothing)\cap\bbl_E)R_u(\bbp_E).$
Hence $\bbh_i\subseteq N_\bbg(R_u(\bbp_\varnothing))=\bbp_\varnothing.$ This implies that $\bbh_i(\bbr)^+\subset Q$ and finishes the proof. 
\end{proof}

Let $\Xi:\bbs\rightarrow \prod_{\alpha\in\Delta}\bbg_m,$
$
\Xi(s):=(\lambda_{\alpha}(s))_{\alpha\in \Delta}.
$
By the discussion in Section \ref{ss:IndexRelativeRootSystem}, it is clear that $\Xi$ is onto and $\ker(\Xi)$ is a finite $\bbq$-group (if $\bbg$ is $\bbq$-split, then it is well-known that $\Xi$ is an isomorphism of $\bbq$-algebraic groups). Hence it induces an isomorphism between $A=\bbs(\bbr)^{\circ}$ and $\prod_{\alpha\in \Delta}\bbr^+$. Hence for any $a\in A$ and $F\subseteq \Delta$, there is a unique $a_F\in A$ such that 
\be\label{e:FProjection}
\lambda_{\alpha}(a_F)=
\begin{cases}
\lambda_{\alpha}(a)&\text{if $\alpha\in F$,}\\
1&\text{if $\alpha\in \Delta\setminus F$.}
\end{cases}
\ee
It will be called the $F$-projection of $a$. By the above definition, it is clear that for any $F\subseteq \Delta$ and $a\in A$, we have 
\be\label{e:Decomposition}
a=a_F\cdot a_{F^c},
\ee
where $F^c=\Delta\setminus F$. By Lemma~\ref{l:CharParabolic} and Lemma~\ref{l:SandP1}, we have that 
\be\label{e:A_FandQ}
a_F\in Q_F,
\ee
for any $a\in A$ and $F\subseteq \Delta$. 
\begin{cor}~\label{c:PartitionIntoRegions} 
For any compactly supported continuous function $\Psi$ on $G/\Gamma$, and positive real numbers $M$ and $\vare'$, there is $(x_{\Psi,\vare',F})_{F\subseteq \Delta}$ such that for any subset $F$ of $\Delta$ we have 
 \begin{equation}~\label{e:measure-estimate} 
 |\int_{G/\Gamma} \Psi\h d(a\mu_Q)-\int _{G/\Gamma}\Psi\h d(a_{F^c} \mu_{Q_F})|<\vare'
 \end{equation} 
 if $\lambda_{\alpha}(a)\ge x_{\Psi,\vare',F}$ for any $\alpha\in F$ and $e^{-M}\le \lambda_{\alpha}(a)\le \max_{|F|<|F'|}(x_{\Psi,\vare',F'})$ for any $\alpha\in \Delta\setminus F$.
\end{cor}
\begin{proof} 
We are essentially dealing with a subset of $A$ whose $F^c$-projection is compact. So the above claim is a direct corollary of Proposition~\ref{p:FindTheLimitUTrans} and a compactness argument.
\end{proof} 
\begin{proof}[Proof of Theorem~\ref{thm;U-trans}]
It is a consequence of Proposition~\ref{p:NoEscape} and Proposition~\ref{p:FindTheLimitUTrans}.
\end{proof} 
\subsection{}\label{ss:GeneralCase}
Retain the notation from Section~\ref{s:statement}. Here we complete the proof of the second part of Theorem~\ref{trans}.

\begin{prop}\label{p:Q_ETranslate}
Let $\{a_n\}\subseteq A$ and $E\subseteq F\subseteq \Delta$. If $\lambda_{\alpha}(a_n)=1$ for any $\alpha\in \Delta\setminus F$ and $\lim_{n\rightarrow \infty} \lambda_{\alpha}(a_n)=\infty$ for any $\alpha\in F\setminus E$, then $a_n\mu_{Q_E}$ converges to $\mu_{Q_F}$ in $\pcal(G/\Gamma)$.
\end{prop}
\begin{proof}
Let $a_{n,E}$ (resp. $a_{n,E^c}$) be the $E$-projection (resp. $(\Delta\setminus E)$-projection) of $a_n$ (see Equation (\ref{e:FProjection})). So we have that $a_n\mu_{Q_E}=a_{n,E^c}\mu_{Q_E}$. Hence without loss of generality we can and will assume that $\lambda_{\alpha}(a_n)=1$ for any $\alpha\in E\cup F^c$. Now let $\psi$ be a compactly supported continuous function on $G/\Gamma$ and let $\vare$ be a given positive real number. For any $n$, by Corollary~\ref{c:PartitionIntoRegions}, one can find $a'_n\in A\cap Q_E$ such that 
\be\label{e:Q_EEstimate}
|\int_{G/\Gamma} (a_n^{-1}\psi) d(a'_n\mu_Q)-\int_{G/\Gamma} (a_n^{-1}\psi) d\mu_{Q_E}|<\vare/2.
\ee
We can further assume that for large enough $n$, $\lambda_{\alpha}(a_n a'_n)\ge x_{\psi,\vare/2,F}$ for any $\alpha\in F$ (notice that $a_n a'_n\in Q_F$). Hence again by Corollary~\ref{c:PartitionIntoRegions} we have that
\be\label{e:Q_FEstimate}
|\int_{G/\Gamma} \psi d(a_n a'_n\mu_Q)-\int_{G/\Gamma} \psi d\mu_{Q_F}|<\vare/2.
\ee
Therefore by (\ref{e:Q_EEstimate}) and (\ref{e:Q_FEstimate}), we have that
\be\label{e:FinalEstimate}
|\int_{G/\Gamma} \psi d(a_n\mu_{Q_E})-\int_{G/\Gamma} \psi d\mu_{Q_F}|<\vare,
\ee
for large enough $n$, which completes the proof of the proposition.
\end{proof}
\begin{cor}\label{c:EPartitionIntoRegions}
For any compactly supported continuous function $\Psi$ on $G/\Gamma$, and positive real numbers $M$ and $\vare'$, there is $(x_{\Psi,\vare',F})_{E\subseteq F\subseteq \Delta}$ such that for any $E\subseteq F\subseteq \Delta$ we have 
 \begin{equation}~\label{e:Emeasure-estimate} 
 |\int_{G/\Gamma} \Psi\h d(a\mu_{Q_E})-\int _{G/\Gamma}\Psi\h d(a_{F^c} \mu_{Q_F})|<\vare'
 \end{equation} 
 if $\lambda_{\alpha}(a)\ge x_{\Psi,\vare',F}$ for any $\alpha\in F\setminus E$ and 
\[ 
 e^{-M}\le \lambda_{\alpha}(a)\le \max\{x_{\Psi,\vare',F'}|\h E\subseteq F'\subseteq \Delta,\h |F|<|F'|\}
\] 
 for any $\alpha\in \Delta\setminus F$.
\end{cor}
\begin{proof}
This is an easy consequence of Proposition~\ref{p:Q_ETranslate} (similar to the proof of Corollary~\ref{c:PartitionIntoRegions}). 
\end{proof}
%%%%%%%%%%%%%%%%%%%% Absolutely continuous case %%%%%%%%%%%%%%%%%%%%%%%%%%%
\subsection{Translates of absolutely continuous measures}\label{ss;Abs-cont}
As we mentioned, combining Theorem~\ref{trans} and Ratner's measure classification theorem one can get a similar statement for translates of certain absolutely continuous measures 
with respect to $\mu_{Q_E}.$ 
 
Let us fix some notation first. For any $F\subseteq \Delta$
let $\bbm'_F=[\bbl_F,\bbl_F]$ where $\bbl_F$ is the Levi subgroup of $\bbp_F.$ 
Let $\bbm'_F=\prod_{i=1}^{k_F}\bbh_i$ be the almost direct product decomposition
of $\bbm'_F$ into $\bbq$-almost simple factors.
We assume these factors are arranged so that 
$\bbh_1,\ldots,\bbh_{l_F}$ denote all the $\bbq$-isotropic factors of
$\bbm'_F$ if there are any such factors, and $l_F=0$ otherwise. 
Then $(\bbm'_F)^+=\prod_{i=1}^{l_F}\bbh_i$ if $l_F>0$ and $(\bbm'_F)^+=\{e\}$ otherwise.
Similarly let $m_F$ 
be so that $\bbh_{l_F+1},\ldots\bbh_{l_F+m_F}$ are all $\bbq$-anisotropic factors of 
$\bbm'_F$ which are $\bbr$-isotropic if there are any such factors, and $m_F=0$ otherwise. 

With this notation and convention, recall from the introduction that $Q_F=M_FR_u(\bbp_F(\bbr))$
where $M_F=\prod_{i=1}^{m_F+l_F}\bbh_i(\bbr)^\circ$ if $m_F+l_F\geq 1$ and $M_F=\{e\}$ otherwise.   

We also note that arguing as in the proof of Proposition~\ref{p:FindTheLimitUTrans} we see that
if $E\subseteq F\subseteq\Delta$ then any $\bbq$-almost simple $\bbq$-anisotropic factor of $M_F$ is a subset of $Q_\varnothing\subseteq Q_E.$ 

\begin{thm}\label{thm;trans-abs-general}
Let the notation and assumptions be as in Theorem~\ref{trans}. Let $E\subseteq\Delta$ and 
let $\mu=f\mu_{Q_E}\in\pcal(G/\Gamma)$ where $f$ is a non-negative 
continuous function on $G/\Gamma$. Then
\begin{enumerate}
\item If $\lambda_{\alpha} (a_n)$ goes to zero for some $\alpha\not\in E$, then $a_n \mu$ 
diverges in $\pcal (G/\Gamma)$.
\item Let the notation be as above further assume that 
\begin{enumerate}
\item Let $E\subseteq F\subseteq \Delta$ be so that if we write $Q_F=M_FR_u(Q_F),$ 
as above, then the adjoint action of $M_F$ on $\Lie(R_u(Q_F))$ has no fixed vector. 
\item With $F$ as in ${\rm (a)},$
suppose $\mu$ is invariant under $\bbh_i(\bbr)^\circ$ for all $l_F+1\leq i\leq l_F+m_F$
if $m_F\neq0.$ 
\item Let $\{a_n\}\subseteq A$ be a sequence such that  
$\lambda_{\alpha}(a_n)=1$ for any $n$ and any $\alpha\not \in F$ and 
$\lambda_{\alpha}(a_n)$ goes to infinity for any $\alpha\in F\setminus E,$ 
\item Suppose $\det (\Ad(a_n)|_{R_u(Q_{\varnothing})\cap H_j})\to\infty$ for all $1\leq j\leq l_F.$
\end{enumerate}
\end{enumerate}  
Then, $a_n\mu\to\mu_{Q_F}.$
\end{thm} 

Comparing Theorem~\ref{trans} and Theorem~\ref{thm;trans-abs-general}:
in what follows with an example we show that already for the case of $\mathbb{G}=\mathbb{SL}_4$ extra conditions are necessary. Similar examples maybe constructed to show that all of these conditions are in some sense necessary. The issue maybe explained as follows: 
in the case when one considers translates of
$\mu_{Q_E}$ this measure is, by definition, $Q_E$-invariant so one does not need to require any conditions on ``dynamics of $a_n$ action on $\Lie Q_E$''. 
However, in the absolutely continuous case, one needs extra conditions on 
the expansion factors of $a_n$ along the leaves that are subsets of 
$\Lie Q_E$ in order to get a homogeneous measure. 
The condition on $F$ is to guarantee the $\mu_{Q_F}$ is $M_F$-ergodic.   

Let $G=\SL_4(\bbr),$ $\Gamma=\SL_4(\bbz),$ $U$ the group of unipotent upper triangular matrices and   
\[
 P_F=\left(\begin{array}{cccc}* & * & * & *\\ * & * & * & *\\ 0 & 0 & * & *\\ 0 & 0 & 0 & *\end{array}\right)\subseteq G.
\]
Take $\{a_n\}$ to be a sequence of diagonal matrices so that
\[
a_n:=\alpha_1^{\vee}(e^n)=\diag(e^n,e^{-n},1,1).
\]
Then $a_n\mu_{U}\to\mu_{Q_F}.$ 
This, however, is not true of $a_n\mu,$ where $\mu=f\mu_{U}$ for an arbitrary 
compactly supported function $f;$ indeed the limiting measure of $a_n\mu$ 
is not a homogeneous measure. 

First note that Theorem~\ref{thm;trans-abs} is a special case of Theorem~\ref{thm;trans-abs-general}.
Indeed the only part which requires a remark is that under the assumption of Theorem~\ref{thm;trans-abs}
we have $\det(\Ad(a_n|_{R_u(Q_\varnothing)}))\to\infty:$ this follows from Lemma~\ref{l:RhoPositiveCone}.

We now turn to the proof of Theorem~\ref{thm;trans-abs-general}. 
The divergence part is clear as under this condition $\mu_{Q_E}$
diverges. So we will focus on part (2);
for simplicity in notation put $Q=Q_E.$ 

Recall that $Q\Gamma/\Gamma$ is a closed orbit 
and $f\mu_{Q}$ is a probability measure on $Q\Gamma/\Gamma.$
Therefore: for any $\vare>0$ there exists a compact subset 
$K_\vare\subseteq G/\Gamma$ so that $\int_{K_\vare^c}f(x)d\mu_Q(x)<\vare.$
Given $\vare>0$ choose $\Phi_\vare\in C_c(G/\Gamma)$ 
so that $\chi_{K_\vare}\leq\Phi_\vare\leq 1.$

Let $\vare>0$ be arbitrary and put $f_\vare=f\Phi_\vare.$ Let $\{a_n\}$ be a sequence as in part (2)
then for any $\psi\in C_c(G/\Gamma)$ we have
\begin{align*}
 \int_{G/\Gamma}\psi(a_nx)f_\vare(x)d\mu_Q(x)&\leq 
 \int_{G/\Gamma}\psi(a_nx)d\mu(x)=
 \int_{G/\Gamma}\psi(a_nx)f(x)d\mu_Q(x)\\
 &=\int_{K_\vare} \psi(a_nx)f(x)d\mu_Q(x)+\int_{K_\vare^c}\psi(a_nx)f(x)d\mu_Q(x)\\
 &\leq \int_{G/\Gamma}\psi(a_nx)f_\vare(x)d\mu_Q(x)+ \vare\|\psi\|_\infty.
\end{align*}
In view of this calculation we may and will assume from now that
$f$ is a continuous compactly supported function on $G/\Gamma.$
Using a partition of unity argument it suffices to show the statement 
for a continuous approximation of characteristic function of small open balls. 
To be more precise, let us fix $q\in\mbox{supp}(\mu_Q)$ and let $B(q,r)=B_Q(e,r)q$ 
where $B_Q(e,r)$ denotes the ball of radius $r$ around $e$ in $Q.$ 
Let  $\delta$ be a positive number which is less than $\frac{1}{10}$-th the injectivity radius 
at $q$ and let $f$ be a nonnegative continuous function bounded by 1 which equals one on $B(q,4\delta)$ and equals 0 on the complement of $B(q,5\delta)$ in $G/\Gamma.$

Indeed the first part is a consequence of the first part of Theorem~\ref{trans}, 
so we will prove the second part.
Let $X=G/\Gamma$ and $\wt{X}$ be its one-point compactification. 
For any $\sigma\in \pcal(X),$ a Borel probability measure, and any measurable function $\varphi$ on $X$, let 
\[
\sigma(\varphi):=\int_X \varphi d\sigma.
\]
\begin{lem}\label{l:GeneralCaseNoEscape}
Let the notation and assumptions be as in Theorem~\ref{thm;trans-abs}. 
If $\mu_{\infty}$ is a limit of a subsequence of $\{a_n\mu\}$ in $\pcal(\wt{X})$, then $\mu_{\infty}(\{\infty\})=0.$
\end{lem}

\begin{proof}
Recall that $\mu=\frac{1}{\mu_Q( f)}fd\mu_Q$ where $f$ is a nonnegative compactly 
supported continuous function as above, in particular $f$ is bounded by $1.$  
Our assumption on $\{a_n\},$ thanks to Theorem~\ref{trans}, implies that $a_n\mu_Q$ 
converges to $\mu_{Q_F}.$ 

Let $\vare>0$ be given and let $\psi\in C_c(X)$ be such that $0\leq\psi\leq 1$ and
$\mu_{Q_F}(\psi)> 1-\vare^2\mu_Q( f)/4.$ 
Suppose $n$ is large enough so that $\mu_Q(a_n^{-1}\psi)>1-\vare^2\mu_Q( f)/2,$ we have 
\begin{align*}
a_n\mu(\psi)=\mu(a_n^{-1}\psi)=\frac{1}{\mu_Q( f)}\int_{G/\Gamma}(a_n^{-1}\psi) f d\mu_Q
\geq\frac{1-\vare}{\mu_Q( f)}\int_{\{x:a_n^{-1}\psi(x)\geq 1-\vare\}}fd\mu_Q
\geq1-2\vare,
\end{align*}
which completes the proof of the lemma.
\end{proof}

Let now $\{a_{n_i}\mu\}$ be a subsequence converging to $\sigma\in\mathcal{P}(X).$ 
\begin{lem}\label{l;general-unip-inv}
There is a unipotent subgroup $V\subseteq Q$ such that $\sigma$ is invariant under $V$
and $\mu_{Q_F}$ is $V$-ergodic. 
\end{lem}

\begin{proof}
Let $\{a_n\}$ be as in the statement of Theorem~\ref{thm;trans-abs-general}.
By assumption (c) we have 
\[
a_n\in \cap_{\alpha\in \Delta\setminus F} \ker \lambda_{\alpha}.
\]
Hence by Lemma~\ref{l:CharParabolic}, $a_n\in \bbp_F^{(1)}$. Since $a_n\in \bbs(\bbr)^{\circ},$ 
by Lemma~\ref{l:SandP1}
we have that $a_n\in \bbp_F^+$. This implies that $a_n\in (\bbm'_F)^+$, where 
$\bbm'_F=[\bbl_F,\bbl_F]$ is as in the beginning of this section. Hence for any $n$ there are 
$a_n^{(j)}\in A\cap \bbh_i$ such that $a_n=a_n^{(1)} \cdots a_n^{(l_F)},$
in particular $l_F\geq 1.$ 

Now assumption ({d}), 
and the above
decomposition imply 
\[
\det (\Ad(a^{(j)}_n)|_{R_u(Q_{\varnothing})\cap H_j})\to\infty\quad
\text{for all $1\leq j\leq l_F.$}
\]
This, in particular, implies that for each $1\leq j\leq l_F$ there exists
a nontrivial subgroup $V_j$
of $R_u(Q_{\varnothing})\cap H_j$ such that  
\be\label{e;contraction}
\mbox{for all $v\in V_j$ we have $(a_n^{(j)})^{-1}va_n^{(j)}\to e$ as $n\to\infty.$}
\ee

Recall that 
$M_F=H_1\cdots H_{l_F+m_F}.$ 
For $1\leq j\leq l_F$ let $V_j\subseteq H_j$ be a nontrivial subgroup so that
~\eqref{e;contraction} holds.  
If $m_F\neq0,$ for every $l_F+1\leq j\leq l_F+m_F$ 
let $V_j$ be a maximal unipotent subgroup of $H_j.$
Put $V=V_1\cdots V_{l_F+m_F}.$ We will show that $V$ satisfies the claim. 

First note that in view of our choice of $F,$ see (a) in the statement,
that $R_u(Q_F)$ is generated by 
$\{u\in R_u(Q_F):\text{ there exists a maximal torus of $M_F$ which expands or contracts $u$}\}.$ 
Therefore, it follows from the generalized Mautner phenomena,
~\cite[Lemma 3]{Mar-ICM}, 
that $\mu_{Q_F}$ is $V$-ergodic.   

%%%%%%%%%%%%%%%%%%%%%%%%%%%%%%%%%%%
%%%%%%%%%%%%%%%%%%%%%%%%%%%%%%%%%%%

We now show $\sigma$ is $V$-invariant.
As was mentioned, it was shown in the course of the proof of Proposition~\ref{p:FindTheLimitUTrans} that if $m_F\neq0,$ then for all 
$l_F+1\leq j\leq l_F+m_F$ we have $H_j\subseteq Q_\varnothing.$ Therefore
for all such $j$ the subgroup $H_j$ centralizes $A.$
This and assumption (b) imply that $\sigma$ is invariant
under $V_j$ for all $l_F+1\leq j\leq l_F+m_F.$ 

In view of the above paragraph, let $v\in V_1\cdots V_{l_F}$ and
put $v_i=a_{n_i}^{-1}va_{n_i}.$ 
Recall that $v_i\rightarrow e$ as $i$ tends to infinity and $\mu_Q$ is invariant under $v_i$ for all $i$. 
Let $\psi\in C_c(X)$ be any nonnegative function. We have 
\begin{align*}
|va_{n_i}\mu(\psi)-a_{n_i}\mu(\psi)|
=&\frac{1}{\mu_Q( f)}\left|\int_X\psi(va_{n_i}x)f(x)d\mu_Q-\int_X\psi(a_{n_i}x)f(x)d\mu_Q\right|\\ 
\leq&\frac{1}{\mu_Q( f)}\int_X\psi(a_{n_i}x)|f(v_i^{-1}x)-f(x)|d\mu_Q,
\end{align*} 
which goes to zero as $i\rightarrow\infty.$ This implies $v\cdot\sigma=\sigma,$ as we wanted. 
\end{proof}

Let now $V$ be any unipotent subgroup of $G.$ 
Let $\mathcal{H}$ denote the collection of closed subgroups $H$ of $G$ 
such that $H\cap\Gamma$ is a lattice in $H$ and the subgroup generated by all one parameter unipotent subgroups contained in $H$ acts ergodically on $H/H\cap\Gamma.$ 
In the case in hand since $\Gamma\subset\bbg(\bbz)$ it follows from the definition 
that $\mathcal{H}$ is a countable set. For any $H\in\mathcal{H}$ let 
\[
T(H,V)=\{g\in G\h|\h Vg\subset gH\},\quad\text{and}
\]
\[
S(H,V)=\bigcup\{T(H',V)|\h H'\in\mathcal{H},\h H'\subset H
\h\mbox{and}\h\dim H'<\dim H\}.
\]
We have $\pi(T(H,V)\setminus S(H,V))=\pi(T(H,V))\setminus \pi(S(H,V)).$
The following fundamental result of M.~Ratner~\cite{Rat2} describes $V$-invariant probability measures on $G/\Gamma.$ 
The following formulation is taken from~\cite{EMS2}.

\begin{thm*}{\rm{(Ratner)}}
Let $\mu$ be a $V$-invariant probability measure on $G/\Gamma.$ For any $H\in\mathcal{H}$ let $\mu_H$ be the restriction of $\mu$ to $\pi(T(H,V)\setminus S(H,V)).$ Then $\mu_H$ is $V$-invariant and any $V$-ergodic component of $\mu_H$ is a $gHg^{-1}$-invariant measure on a closed orbit $g\pi(H)$ for some $g\in T(H,V)\setminus S(H,V).$ In particular for any Borel measurable set $B\subset G/\Gamma$ we have
\[\mu(B)=\sum_{H\in\mathcal{H}^*}\mu_H(B),\]
where $\mathcal{H}^*$ is a countable subset of $\mathcal{H}$ consisting of one representative for each $\Gamma$-conjugacy class of elements in $\mathcal{H}.$ 
\end{thm*}

We now continue the proof of Theorem~\ref{thm;trans-abs}. 
Let the notation be as before and note that $G\in\mathcal{H}.$ 
Let us recall that $a_{n_i}\mu\rightarrow\sigma$ and $\sigma$ is $V$-invariant. 
Thanks to the above seminal theorem and the fact that $\mu_{Q_F}$ is $V$-ergodic 
we need to show $\mu(S(G,V))=0.$ 

Let $H\subseteq Q_F$ be in $\mathcal{H}$ with $\dim H<\dim Q_F.$ 
Let $\psi$ be a continuous compactly
supported function on $T(H,V)$ then
\[
a_{n_i}\mu( \psi)=\mu(a_{n_i}^{-1}\psi)\leq
\frac{1}{\mu_Q(f)}\mu_Q(a_{n_i}^{-1}\psi|_{B(q,5\delta)})\leq
\frac{1}{\mu_Q(f)}\mu_Q(a_{n_i}^{-1}\psi).
\]
Now $a_{n_i}\mu_Q\rightarrow\mu_{Q_F}$ and $\mu_{Q_F}(\psi)=0$ imply that 
$\mu_Q(a_{n_i}^{-1}\psi)\rightarrow 0.$ Thus we get  $\sigma(T(H,V))=0,$ which finishes the proof.

\section{Counting horospheres.} \label{s:CountingHorosphers}
Eskin and McMullen~\cite{EM} considered the following counting problem, among others. Let $x_0$ be a fixed point in $\mathbb{H}^2$ the two dimensional hyperbolic space, $c$ a horocycle, and $\Gamma$ a lattice in $G={\rm PSL}_2(\bbr)$ such that $C=\pi(c)$ is a closed horocycle in $\Sigma=\mathbb{H}^2/\Gamma$. Then $N(R)$ the number of $\Gamma$ translates of $c$ which intersects the ball $B(x_0,R)$ of radius $R$ centered at $x_0$ is asymptotically \[\frac{1}{\pi}\frac{{\rm length}(C)}{{\rm area}(\Sigma)}{\rm area}(B(x_0,R)).\]
They also remark that their method does not work for the maximal unipotent subgroups in higher-rank groups. Indeed they used mixing and a key lemma called {\em wavefront lemma} which can show a special case of Theorem~\ref{trans}. Namely, if $a_n$'s go to infinity in the positive Weyl chamber (instead of the interior of the dual of the Weyl chamber), then $\{a_n\mu_U\}$ converge to the Haar measure of $G$. In this section, we address the higher rank version of this problem and prove Theorem~\ref{t:CountingLifts}. With the same notation as in Theorem \ref{t:CountingLifts}, let
\[
N( R)=\#\hspace{.5mm}\{\gamma\in\Gamma\h|\h\mathcal{U}\gamma\cap B(x_0,R)\neq\varnothing\}.
\]  
Using the Iwasawa decomposition we can find $a_0\in A$ such that  $Ug_0K =U a_0 K$.

 Without loss of generality we can assume that $G$ is the real points of a simply connected algebraic $\bbr$-group. 
 Let $\widetilde{B}_R=\{g\in G\hspace{1mm}|\hspace{1mm}d(x_0, x_0g)\le R\}.$  
 Put $A_R=A\cap\widetilde{B}_R$, and $\overline{B}_R=\{gU\hspace{1mm}|\hspace{1mm}g\in\widetilde{B}_R\}$. 
Define $F_R$ as follows 
 \[
 F_R(g\Gamma):=\#\{\gamma\in\Gamma: \mathcal{U}\gamma\cap B(x_0,R)g\neq\varnothing\};
 \]
then $N( R)=F_R(\Gamma).$ 
Since the stabilizer of $\mathcal{U}$ is $U$ and in view of the fact that $\widetilde{B}_R={\widetilde{B}_R}^{-1}$
we have
\[
F_R(g\Gamma)=\sum_{\gamma\in\Gamma/\Gamma\cap U} \1bf_R(g\gamma a_0 U),
\] 
 where $\1bf_R$ is the characteristic function of $\overline{B}_R$. Let $\widetilde{F}_R(g\Gamma)= \frac{1}{f_{a_0}(R)}\1bf_R(g\Gamma)$, where $f_{a_0}(R)=\int_{A^{>}_Ra_0}\rho'_{\Delta}(a) da$ and $A^{>}_Ra_0=\{a\in A_Ra_0\hspace{1mm}|\hspace{1mm}\lambda_i(a)\ge 1,\h \forall i\}$. We start by describing the asymptotic growth of $f_{a_0}(R)$.
\begin{lem}~\label{asymptotic1}
Let ${\bf v}_0\in \bbr^n$ and $\ccal \subseteq \bbr^n$ be an open convex cone centered at the origin such that ${\bf v}_0\in \ccal$. Then there is a constant $c$, such that for any ${\bf x}_0\in \bbr^n$
\[
\int_{\begin{array}{l}
\|{\bf y}\|\le R\\
{\bf y}-{\bf x}_0\in \ccal
\end{array}} e^{\langle {\bf v_0},{\bf y}\rangle} d{\bf y}\sim
\int_{\|{\bf y}\|\le R} e^{\langle {\bf v_0},{\bf y}\rangle} d{\bf y}\sim
cR^{(n-1)/2}e^{\|{\bf v}_0\|R}
\]
\end{lem}
\begin{proof} See \cite[Lemma 9.4]{GW}.
\end{proof} 
\begin{lem}~\label{asymptotic2}
For any non-zero vector ${\bf v}_0\in \bbr^n$, we have
\[
\int_{\|{\bf y}\|\le R} e^{\langle{\bf v}_0,{\bf y}\rangle} d{\bf y}\sim \left(\frac{2\pi R}{\|{\bf v}_0\|}\right)^{\frac{n-1}{2}}\cdot e^{\|{\bf v}_0\|R},
\] as $R$ goes to infinity. 
\end{lem}
\begin{proof} By Fubini's theorem, 
\begin{equation}~\label{fubini}
\int_{\|{\bf y}\|\le R} e^{\langle{\bf v}_0,{\bf y}\rangle} d{\bf y}=\int^{\|{\bf v}_0\| R}_{-\|{\bf v}_0\| R}\vol\{{\bf y}| \|{\bf y}\|\le R, \langle{\bf v}_0,{\bf y}\rangle=s\}\cdot e^{s} ds.
\end{equation} 
On the other hand, intersection of the  hyperplane  $\langle{\bf v}_0,{\bf y}\rangle=s$ and the ball of radius $R$ centered at the origin is a ball of dimension $n-1$ and of radius $\sqrt{R^2-\frac{s^2}{\|{\bf v}_0\|^2}}$. Hence the right hand side of Equation~(\ref{fubini}) is equal to 
\[
\nu_{n-1}\int^{\|{\bf v}_0\| R}_{-\|{\bf v}_0\| R} (R^2-\frac{s^2}{\|{\bf v}_0\|^2})^{\frac{n-1}{2}}\cdot e^{s} ds=\nu_{n-1}\|{\bf v}_0\| R^n\int^{1}_{-1} (1-s^2)^{\frac{n-1}{2}} e^{\|{\bf v}_0\| R s} ds,
\] 
where $\nu_{n-1}$ is the volume of a ball of radius $1$ in $\bbr^{n-1}$. Let $g_m(x)=\int^{1}_{-1} (1-s^2)^{\frac{m}{2}} e^{xs} ds.$ Using integration by part, it is easy to see that \begin{equation}~\label{recursive1}x^2 g_m(x)=-m(m-1) g_{m-2}(x)+m(m-2) g_{m-4}(x),\end{equation} 
for $m\ge 3$. Let $\overline{g}_m(x)=x^m g_m(x)/ \varpi_{m}$, where $\varpi_m=\prod_{0\le k<m/2}(m-2k)$. So Equation~(\ref{recursive1}) gives us 
\begin{equation}~\label{recursive2} \overline{g}_m(x)=-(m-1) \overline{g}_{m-2}(x)+x^2 \overline{g}_{m-4}(x).
\end{equation} 
By induction, it is clear that there are sequences of polynomials $P_m$ and $Q_m$ which satisfy similar recursive formulas as in equation~\ref{recursive2}, and 
\[\overline{g}_m(x)=\left\{ \begin{array}{l}P_m(x) \overline{g}_2(x)+Q_m(x) \overline{g}_0(x)\hspace{1mm}{\rm if}\hspace{1mm}2|m.\\ 
\\P_m(x) \overline{g}_1(x)+Q_m(x) \overline{g}_{-1}(x)\hspace{1mm}{\rm if}\hspace{1mm}2\not | m.\end{array}\right. \] We can further determine the leading terms of $P_m$ and $Q_m$.

\[\mbox{leading terms of}\hspace{1mm} P_m \& Q_m=\left\{
\begin{array}{lll}-k(2k+1) x^{2k-2},&x^{2k}&{\rm if}\hspace{1mm}m=4k.\\
\\
x^{2k},&-2k(k+1) x^{2k}&{\rm if}\hspace{1mm}m=4k+1.\\
\\
x^{2k},&-k(2k+3) x^{2k}&{\rm if}\hspace{1mm}m=4k+2.\\
\\
-2(k+1)^2 x^{2k},&x^{2k+2}&{\rm if}\hspace{1mm}m=4k+3.
\end{array}
\right.
\]
One also can easily see that $\overline{g}_0(x)\sim \frac{e^x}{x}$ and $\overline{g}_2(x)\sim e^x$. By the definition of modified Bessel functions, one can see that $g_1(x)=\frac{\pi}{x} I_1(x)$ and $g_{-1}(x)=\pi I_0(x)$, where $I_0$ and $I_1$ are modified Bessel functions. Hence $\overline{g}_1(x)\sim \pi\frac{e^x}{\sqrt{2\pi x}}$ and $\overline{g}_{-1}(x)\sim \frac{\pi}{x}\cdot \frac{e^x}{\sqrt{2\pi x}}$ as $x$ goes to infinity. Altogether 
\[
\overline{g}_m(x)\sim 
\left\{
\begin{array}{ll}
x^{2k-1} e^x&{\rm if}\hspace{1mm} m=4k.\\
\\
\sqrt{\frac{\pi}{2}} x^{2k-\frac{1}{2}} e^x&{\rm if}\hspace{1mm} m=4k+1.\\
\\
x^{2k} e^x&{\rm if}\hspace{1mm} m=4k+2.\\
\\
\sqrt{\frac{\pi}{2}} x^{2k+\frac{1}{2}} e^x&{\rm if}\hspace{1mm} m=4k+3.\\
\end{array}
\right.
\]
Therefore combine it with the formula of $\nu_m$ one can finish the proof: 
\[
\nu_{n-1}\|{\bf v}_0\| R^n g_{n-1}(\|{\bf v}_0\|R) \sim \left(\frac{2\pi R}{\|{\bf v}_0\|}\right)^{\frac{n-1}{2}}\cdot e^{\|{\bf v}_0\|R}.
\] \end{proof}
Let $\afr$ be the Lie algebra of $A$. It can be also viewed as $X^*(\bbs)\otimes \bbr$ (see Section \ref{ss:IndexRelativeRootSystem}). The negative of the Killing form turns $\afr$ into a Euclidean space and moreover one can choose the Riemannian metric on $X$ such that $d(x_0,x_0\exp(a))=\|a\|$ for any $a\in \afr$. For any $F\subseteq \Delta$, let $\rho_F$ be a vector in $\afr$ such that 
\be\label{e;rho-rho}
\rho'_F(\exp(a))=e^{(\rho_F, a)},
\ee
for any $a\in \afr$.
\begin{lem}\label{l:asymptotic3}
In the above setting, we have
\[
f_{a_0}(R)\sim \left(\frac{2\pi R}{\|\rho_{\Delta}\|}\right)^{\frac{n-1}{2}}\cdot e^{\|\rho_{\Delta}\|R}
\]
for any $a_0\in A$.
\end{lem}  
\begin{proof}
This is a direct corollary of Lemmas~\ref{asymptotic1},~\ref{asymptotic2}.
\end{proof}
\begin{lem}\label{l:Compare}
In the above setting, for any proper subset $F$ of $\Delta$, we have 
\[
\lim_{R\rightarrow \infty} \frac{\int_{A_{F,R}}\rho'_{\Delta}(a_F)da_F}{f_{a_0}(R)}=0,
\]
where $A_{F,R}:=\{a\in A|\h \forall\h\alpha\in \Delta\setminus F,\h \lambda_{\alpha}(a)=1, \h d(x_0,x_0a_0a)\le R\}$. 
\end{lem}
\begin{proof}
It is a direct consequence of Lemmas~\ref{asymptotic1}, \ref{asymptotic2} and the fact that for any $\alpha\in \Delta$, $\langle \lambda_{\alpha},\rho_{\Delta}\rangle\neq 0$.
\end{proof}
Next lemma gives a decomposition for $\overline{B}_R$. 
\begin{lem}~\label{decom} $\overline{B}_R=KA_RU/U$.
\end{lem}
\begin{proof} Let $gU\in \overline{B}_R$. By Iwasawa decomposition, without loss of generality, we can and will assume that $g=k a$, where $k\in K$ and $a\in A$. Let $u\in U$ such that $d(x_0,x_0k a u )\le R$, therefore $d(x_0, a u x_0)\le R. $
 
Let $a_t x_0$ be a geodesic in $A x_0$, and 
\[\mbox{dist}(t)=d(x_0a_t , x_0a_t au )=d(x_0, x_0a (a_t u a_t^{-1})).\] 
So choosing the contracting direction, we get $\lim_{t\rightarrow-\infty} \mbox{dist}(t)=d(x_0, x_0a)$. 
On the other hand, $\mbox{dist}$ is a convex function~\cite[Theorem 3.6]{Bu} and so its value at each point is at least the value of its limit in the infinity if finite. Therefore for any $t$, $d(x_0a_t , x_0a_t au )\ge d(x_0, x_0a)$. In particular, 
$ d(x_0, x_0a u )\ge  d(x_0, x_0a)$, and so $d(x_0, x_0a)\le R$, which finishes proof of the lemma. 
\end{proof}

For any subset $F$ of $\Delta,$ and any $a\in A$, let $a=a_F\cdot a_{F^c}$ be as in~\eqref{e:Decomposition}. Recall that $\lambda_{\alpha}(a_F)=\lambda_{\alpha}(a)$ for $\alpha\in F$ and is 1 for $\alpha\not\in F$ and $a_F\in Q_F.$ 
\begin{lem}~\label{weak}
For any compactly supported continuous function $\Psi$ on $G/\Gamma$, 
\[ 
\langle \widetilde{F}_R,\Psi\rangle\rightarrow \frac{\vol(U/U\cap\Gamma)}{\vol(\mathcal{M})} \langle 1,\Psi\rangle,
\] 
as $R$ goes to infinity.
\end{lem}

\begin{proof} Without loss of generality we can assume that $ \langle 1,\Psi\rangle=1$. As before $\nu$ denotes the measure coming from the Riemannian geometry and $\mu$ is a probability measure on the corresponded space.
\[
\begin{array}{rl}\langle \widetilde{F}_R,\Psi\rangle=&\frac{1}{f_{a_0}(R)}\int_{G/\Gamma}\sum_{\gamma\in\Gamma/\Gamma\cap U} \1bf_R(g\gamma a_0 U) \overline{\Psi(g\Gamma)}\hspace{1mm} d\nu_G\\
 & \\
=&\frac{1}{f_{a_0}(R)}\int_{G/\Gamma\cap U}\1bf_R(ga_0U) \overline{\Psi(g\Gamma)}\hspace{1mm} d\nu_G\\ 
& \\
=& \frac{\nu(\pi(U))}{f_{a_0}(R)}\int_{G/U}\int_{U/\Gamma\cap U}\1bf_R(gua_0U) \overline{\Psi(gu\Gamma)}\hspace{1mm} d\mu_U(u) d\nu_G(g)\\
 & \\
=& \frac{\nu(\pi(U))}{f_{a_0}(R)}\int_{G/U}\1bf_R(ga_0U)\int_{G/\Gamma} \overline{\Psi(g'\Gamma)}\hspace{1mm} d(g\mu_U)(g') d\nu_G(g)\\ 
& \\
\leftexp{\rm (Lemma~\ref{decom})\leadsto}{=}& \frac{\nu(\pi(U))}{f_{a_0}(R)}\int_{K}\int_{A_Ra_0}\int_{G/\Gamma} \overline{\Psi(g'\Gamma)} \hspace{1mm} d(ka\mu_U)(g')\rho'_{\Delta}(a) da\hspace{1mm} dk.
\end{array}
\]
Let $C$ be a compact subset of $G$ such that $\pi( C)=\mbox{supp}(\Psi)$.  If $\mbox{supp}(ka\mu_U)\cap\mbox{supp}(\Psi)\neq\emptyset$, then $kau=c\gamma$ for some $c\in C$ and $\gamma\in \Gamma$, and by Lemma~\ref{out}, there is a constant $M$ depending on $\Psi$ such that for any $\alpha$, $|\lambda_\alpha(a)|\ge e^{-M}$. Now, by Corollary~\ref{c:PartitionIntoRegions}, we can approximate the limiting measures with an error less than $\vare'$ with respect to the test function $\Psi$ (see Figure~\ref{F:Partition}). 
\begin{figure}
\includegraphics[width=\columnwidth]{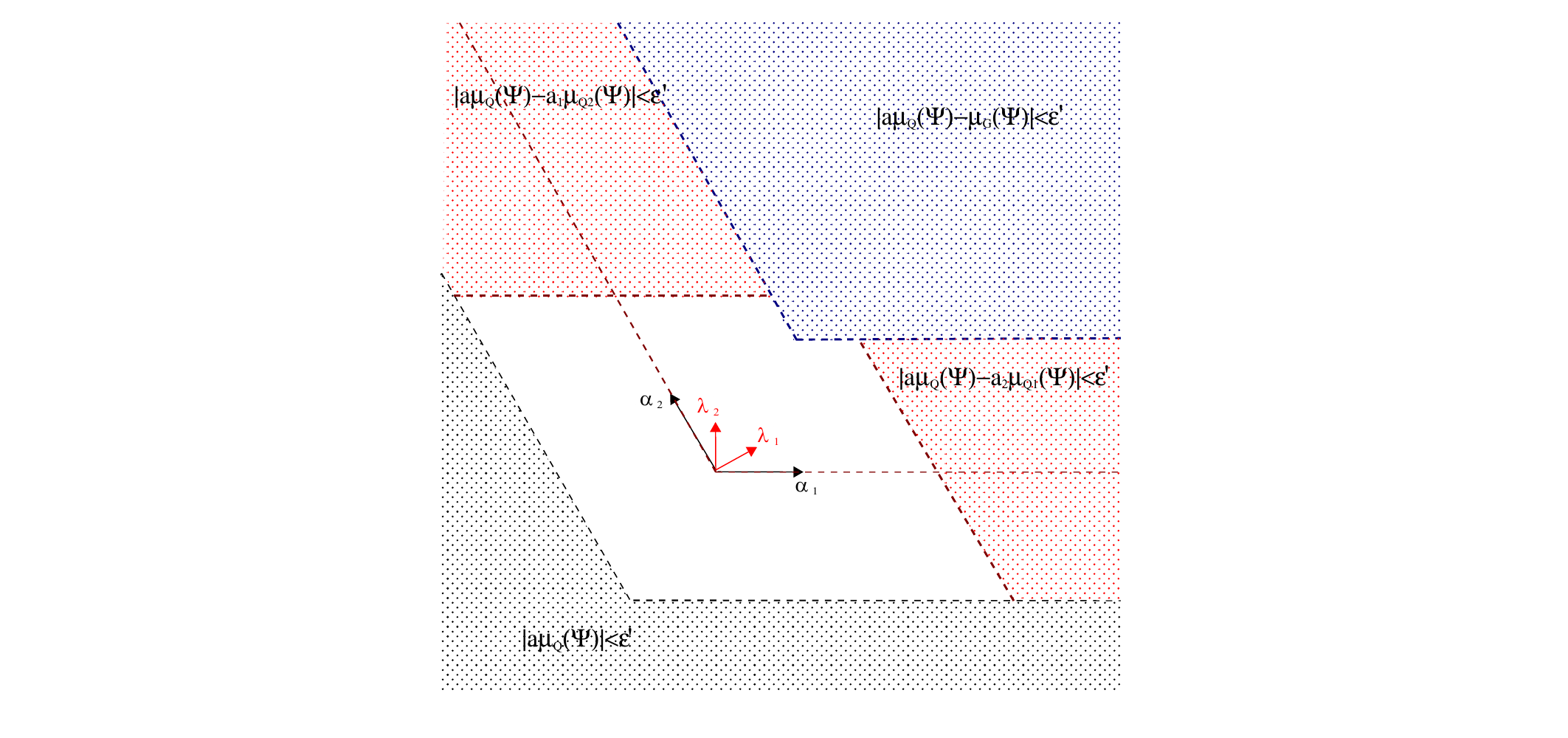}
\caption{The $\mathbb{SL}_3$ case}
\label{F:Partition}
\end{figure}
This gives us a partition $\{T_{\Psi,\vare',F}\}_{F}$ of $\{a\in A| \lambda_{\alpha}(a)\ge e^{-M}\}$ such that 
\begin{enumerate}
\item $|(ka\mu_U)(\Psi)-(ka_{F^c}\mu_{Q_F})(\Psi)|<\vare'$ for any $k\in K$ and $a\in T_{\Psi,\vare',F}$.
\item The $F^c$-projection ${\rm pr}_{F^c}(T_{\Psi,\vare',F})$ of $T_{\Psi,\vare',F}$ is compact.
\end{enumerate}
 So by the above equations $\langle \widetilde{F}_R,\Psi\rangle$ is equal to
\begin{equation}~\label{partition}\frac{\nu(\pi(U))}{f_{a_0}(R)}\int_{K}\sum_{F\subseteq\Delta}\int_{A_Ra_0\cap T_{\Psi,\vare',F}}\int_{Q_F\Gamma/\Gamma} \overline{\Psi(ka_{F^c}q_F\Gamma)} \hspace{1mm} d\mu_{Q_F}\rho'_{\Delta}(a) da\hspace{1mm} dk\end{equation}
\[\pm\frac{\vare'\cdot \nu(\pi(U))}{f_{a_0}(R)}\cdot\vol(K)\cdot\int_{\{a\in A_Ra_0|\forall \alpha, \lambda_\alpha(a)\ge e^{-M}\}}\rho'_{\Delta}(a) da.
\]
By Lemmas~\ref{asymptotic1}, \ref{asymptotic2}, it is clear that the second term is of order $O(\vare')$ for large enough $R$. In particular, we can make it as small as we wish by decreasing $\vare'$, and increasing $R$. So for a given $F$, only the first part should be considered. 

For any subset $F$ of $\Delta$ and any $a\in A_R a_0\cap T_{\Psi,\vare',F}$, let $A_{R,F}:=A_Ra_0\cap T_{\Psi,\vare',F}$ and
\[
A(a_{F^c}):=\{{\rm pr}_F(a')\h|\h {\rm pr}_{F^c}(a')=a_{F^c},\h a'\in A_{R,F}\}.
\]
So the $F$-summand of Equation (\ref{partition}) is equal to 
\[\int_{K}
\int_{\mbox{pr}_{F^c}(A_{R,F})}
\int_{Q_F\Gamma/\Gamma} \overline{\Psi(ke^{\xbf_{F^c}}q_F\Gamma)}
\rho'_{\Delta}(a_{F^c}) 
d\mu_{Q_F} 
da_{F^c} 
\int_{A(a_{F^c})} 
\rho'_{\Delta}(a_F)da_{F}
\hspace{1mm} dk.\]
By Lemma~\ref{l:Compare} and the fact that $\mbox{pr}_{F^c}(A_{R,F})$ is compact, if $F$ is a proper subset of $\Delta$. So by Lemma \ref{l:Compare} the $F$-summand gets arbitrarily small as $R$ goes to infinity. Overall we have
\[\begin{array}{rl}\langle \widetilde{F}_R,\Psi\rangle=&
\frac{\nu(\pi(U))}{f_{a_0}(R)}
\int_{K}
\int_{A_Ra_0\cap T_{\Psi,\vare',\Delta}}
\int_{\pi(Q_{\Delta})} \overline{\Psi(ka_Fq_{\Delta}\Gamma)} \hspace{1mm} d\mu_{Q_{\Delta}}
\rho'_\Delta(a) da\hspace{1mm} dk\\
&\\
+\vare''\sum_{F\neq\Delta }&
\int_{K}
\int_{\mbox{pr}_{F^c}(\mathfrak{a}_{R,F})}
\int_{\pi(Q_F)} \overline{\Psi(ke^{\xbf_{F^c}}q_F\Gamma)}
\hspace{1mm}e^{\sum_{\alpha\not\in F}m_\alpha x_\alpha} 
d\mu_{Q_F} 
d\xbf_{F^c}
\hspace{1mm}dk \\
&\\
\pm&O(\vare'), 
\end{array}\]
for large enough $R$. In the second term, $\mbox{pr}_{F^c}(\mathfrak{a}_{R,F})$ is a bounded region. Therefore the integral is bounded by a function of $a_0$, $\vare'$, and $\Psi$. So by fixing them and increasing $R$, we can make the second term as small as we wish. Thus only the first term should be studied. Since $G$ is a semisimple Lie group without compact factors $Q_{\Delta}=G$. Hence the first term equals to
$\frac{\nu(\pi(U))}{\vol(G/\Gamma)}\vol(K)\cdot\frac{1}{f_{a_0}(R)}\int_{A_Ra_0\cap T_{\Psi,\vare',\varnothing}}\rho(a)^2 da.$ We get the desired result by again using Lemma~\ref{asymptotic1}.
\end{proof}
\noindent
We now prove the pointwise convergence. Similar to~\cite{DRS}, it is enough to take ${\Psi}_{\vare}$ an approximation of the identity near $\pi(g)$, and then study what happens to $\widetilde{F}_R(g)$ after a small perturbation of $g$. Let $\Psi_{\vare}$ be such that 
\[\mbox{supp}(\Psi_{\vare})\subseteq\{\pi(g')| d(x_0g,x_0g')\le \vare\hspace{1mm} \&\hspace{1mm} d(x_0g^{-1},x_0g'^{-1})\le \vare\}.\] 
Then it is easy to see that 
\begin{equation}~\label{pointwise}
\frac{f_{a_0}(R-\vare)}{f_{a_0}(R)}\cdot\widetilde{F}_{R-\vare}(\pi(g))\le
\langle \widetilde{F}_R,\Psi_{\vare}\rangle\le
\frac{f_{a_0}(R+\vare)}{f_{a_0}(R)}\cdot\widetilde{F}_{R+\vare}(\pi(g)).
\end{equation}
So the following lemma together with Lemma~\ref{weak} and Equation~(\ref{pointwise}) show that $\widetilde{F}_R(\pi(g))$ tends to $\frac{\vol(U/U\cap\Gamma)}{\vol(\mathcal{M})}$, as $R$ goes to infinity, for any $g\in G$.  
\begin{lem}~\label{wellroundness!}
We have
\[b(\vare)=\lim_{R\rightarrow \infty} \frac{f_{a_0}(\vare+ R)}{f_{a_0}(R)},\]
with $b(\vare)\rightarrow 1$ as $\vare\rightarrow 0$.
\end{lem}
\begin{proof} This is a direct corollary of  Lemmas~\ref{asymptotic1} and~\ref{asymptotic2}.
\end{proof}

\section{Batyrev-Tschinkel-Manin's conjecture for flag variety.}
In this section, we prove Theorem \ref{t:Flag}. To do so, first we prove it 
for the anti-canonical line-bundle and then deal with the general case. Though 
the general approach is similar to the previous section, there are several 
technical differences, specially for an arbitrary metrized line-bundle.

As we mentioned in Section \ref{ss:RationalPoints}, by a theorem of Borel-Harish-Chandra, there is $\Xi,$ a finite subset of $\bbg(\bbq),$ such that $\bbg(\bbq)=\bbp(\bbq)\cdot\Xi\cdot\Gamma$, where $\Gamma=\bbg(\bbz)$. Also $(\bbg/\bbp_E)(\bbq)=\bbg(\bbq)/\bbp_E(\bbq),$ for any $E\subseteq\Delta,$~\cite[Lemma 2.6]{BT}.  
So it is enough to understand the asymptotic behavior of 
\[
N_T=\#\{\gamma\in\Gamma/\Gamma\cap \bbp_E(\bbr)\h|\h \|\eta_\chi(\gamma)v\|\le T\},
\]
where $\chi$ is the highest weight of an irreducible representation 
$\eta_{\chi}:\bbg\rightarrow \mathbb{GL}(\bbv)$, $\|\h\|$ 
is a $K$-invariant norm on $\mathbb{V}(\bbr)$ and $\|v\|=1.$ 

Recall $\bbp_E^{(1)}=\cap_{\lambda\in X^*(\bbp_E)_{\bbq}}\ker(\lambda),$ put $\peone=\bbp_E^{(1)}(\bbr).$
We note that $\peone\cap\Gamma$ is a lattice in $\peone,$ 
we will denote by $\mu_{\peone}$ the $\peone$-invariant probability
measure on $\peone/\peone\cap\Gamma.$ 

Recall that we fixed a $\bbq$-torus $\bba'\subseteq\bbp_\varnothing$ so that $\bba'(\bbr)^\circ=A'$ is a 
maximal $\bbr$-split torus. 
As was noted before $Q_E$ is a normal subgroup of $\peone.$ 
We have
$
\peone=\mathbb{Y}_E(\bbr)\bba_E''(\bbr)Q_E
$
where 
\begin{itemize}
 \item $\mathbb{Y}_E=(\prod_{J_E}\bbh_j)\mathbb{D}_E$ where $\bbh_j$ is a $\bbq$-almost simple group which is a $\bbr$-anisotropic and $\mathbb{D}_E$ is a $\bbq$-torus which is $\bbr$-anisotropic, and 
 \item $\bba''_E\subseteq \bba'$ is a $\bbq$-anisotropic torus which is $\bbr$-split.
\end{itemize}
Moreover, the product is an almost direct product. 
In particular, $\mathbb{Y}_E(\bbr)$ is a compact group. Also note that $\mathbb{Y}_E$ and $\bba''$ centralize $\bba$.  
Put $Y_E=\mathbb{Y}_E(\bbr)$ and $A_E''=\bba_E''(\bbr).$

Note that if $E\subseteq F,$ then $\mathbb{Y}_F\bba''_F\subseteq\mathbb{Y}_E\bba''_E.$
To see this first note that $\bba''_F\subseteq\bba''_E,$ because $\bba''_F\subseteq\bbp_E$
and it is a central component of Levi subgroup of $\bbp_F.$ Thus $\bba''_F$ is a central component
of the Levi subgroup of $\bbp_E;$ this implies $\bba''_F\subseteq\bba''_E.$
To see the inclusion $\mathbb{Y}_F\subseteq\mathbb{Y_E}$
one argues as in the proof of Proposition~\ref{p:FindTheLimitUTrans}.
Note also that the above product decomposition implies that $\mathbb{Y}_F$ is
also a factor of $\mathbb{Y}_E;$ the same is true of $A''_F$ and $A''_E$
as they are tori.

Recall our notation $Q_E=M_ER_u(Q_E).$ Put $B_E=Y_EA''_EM_E$
and let
\[
\mbox{$\mathfrak{p}:\peone\to\peone/R_u(Q_E)=B_E$}
\] 
be the natural projection.
Abusing the notation let us also denote the induced fiberation 
$\peone/\peone\cap\Gamma\to B_E/B_E\cap\Gamma$ by $\mathfrak{p}.$

The group $B_E$ is an almost direct product of commuting factors
$Y_E, A_E, M_E$ which intersect $\Gamma$ in lattices, therefore, the $B_E$
invariant probability measure on $B_E/B_E\cap\Gamma$ has a product 
decomposition into the corresponding probability measure.     
Altogether we get:
for any $\psi\in C_c(G/\Gamma)$ we have
\begin{align}
\label{e;Fubini}\int_{G/\Gamma}\psi d\mu_{\peone}&=\int_{G/\Gamma} \psi d\mu_{B_E}d\mu_{R_u(Q_E)}\\&
\notag=\int_{G/\Gamma}\psi d\mu_{Y_E}d\mu_{A''_E}d\mu_{M_E}d\mu_{R_u(Q_E)}
=\int_{G/\Gamma}\psi d\mu_{Y_E}d\mu_{A''_E}d\mu_{Q_E}.
\end{align}

Let $B_T$ be the ball of radius $T$ centered at the origin in $\mathbb{V}(\bbr)$, 
and similar to the previous section let 
$\widetilde{B}_T$ be the pull back of $B_T$ in $G$, 
and $\overline{B}_T$ be the image of $\widetilde{B}_T$ in $G/\peone$. 
This time, it is much simpler than the previous section to give a decomposition of $\overline{B}_T$. 
With the notation as in~\eqref{e:Decomposition} (where we defined the $F$-component $a_F$ of $a\in A$ for $F\subseteq \Delta$ and showed that $a=a_F\cdot a_{F^{c}}$) we have the following.
\begin{lem}~\label{B-decomposition}
Let $A_{E^c,T}=\{a\in A\h|\h a=a_{E^c}\h,\h\chi(a)\le T\}.$ Then 
\[
\overline{B}_T=K A_{E^c,T} \peone/\peone.
\] 
\end{lem} 

\begin{proof} Recall that $\|\h\|$ is $K$-invariant, $\peone$ is the stabilizer of the vector $v$ and $A$ acts by the character $\chi$ on $v$ which indeed factors through the $E^c$ component. Now the lemma follows from Langlands decomposition of $\bbp_E(\bbr)$, see for example~\cite{Kn}.
\end{proof}

Let $\1bf_T$ denote the characteristic function of $B_T$ and define 
\[
F_T(g\Gamma):=\sum_{\gamma\in\Gamma/\Gamma\cap\bbp_E(\bbr)} \1bf_T(\eta(g\gamma)v).
\]
Let $\widetilde{F}_T(g\Gamma)=\frac{1}{f(T)}F_T(g\Gamma)$, where 
\[f(T)=\int_{A^+_{E^c,T}} \rho'_E(a) da\h\h{\rm and}\h\h A^+_{E^c,T}=\{a\in A_{E^c,T}\h|\h\forall \alpha,\h \lambda_\alpha(a)\ge 1\}
.\] 
We will prove that $\widetilde{F}_T(e)$ tends to a constant as $T$ goes to infinity. In the analogues case in the previous section, we proved a much stronger statement. We proved both weak and pointwise convergence of $\widetilde{F}_R$ to a constant function. However in this section, we cannot prove such statements. Instead we only prove what we need for the counting problem, namely pointwise convergence at the identity.

Let $\chi=\sum_{\alpha\in E^c}c_\alpha \lambda_\alpha,$ let us emphasize again that we denote the logarithm of characters of $A$ with the same notation as the characters themselves. And by Lemma~\ref{l:RhoPositiveCone} we have $\rho'_E=\sum_{\alpha\in E^c}m_\alpha\lambda_\alpha$, where $m_{\alpha}\in \bbz^+$. In particular, if $\bbg$ is $\bbq$-split and $E=\varnothing$, then $m_{\alpha}=2$ for any $\alpha\in \Delta$.

We identify $\bbr^{|\Delta|}$ with $A$ (see (\ref{e:Decomposition})) via the isomorphism $\Theta:\bbr^{|\Delta|}\rightarrow A$ such that 
\be\label{e:Theta}
 \left(\log(\lambda_{\alpha}(\Theta(\xbf)))\right)_{\alpha\in \Delta}=\xbf,
\ee
for any $\xbf\in \bbr^{|\Delta|}$. Similarly using $\Theta$ we identify $\vcal_F:=\{\xbf=(x_{\alpha})_{\alpha\in\Delta}|\h \forall\h\alpha\in\Delta\setminus F,\h x_{\alpha}=0\}$ with $A_F$, for any subset $F$ of $\Delta$. 

We can and will assume that the Haar measure $da_F$ of $A_F$ is the push-forward $\Theta_*(d\xbf_F)$ of the Lebesgue measure of $\vcal_F$, i.e. for any integrable function $g\in L^1(A_F)$ we have 
\[
\int_{A_F} g(a_F) da_F=\int_{\vcal_F} g(\Theta(\xbf_F)) d\xbf_F.
\] 
In particular for any bounded set $B\subseteq \vcal_{E^c}$ we have
$
\int_{\Theta(B)} \rho_E'(a)da=\int_{B} e^{\sum_{\alpha\in E^c}m_{\alpha}x_{\alpha}} d\xbf_{E^c}. 
$
For any $F\subseteq \Delta$, let ${\rm pr}_F:\bbr^{|\Delta|}\rightarrow \vcal_F$ be the (natural) projection onto $\vcal_F$.
 
\begin{definition}\label{d:PositiveLinearFunctional}
\begin{enumerate}
  \item For $F\subseteq \Delta$, let $\vcal_F^+:=\{\xbf\in \vcal_F\setminus \{0\}|\h\forall\h\alpha\in F,\h x_{\alpha}\ge 0\}$.
	\item A linear function $l:\vcal_F\rightarrow \bbr$ is called a {\em positive linear functional} if $l(\vcal_F^+)=\bbr^+$.
	\item Let $l_F:\vcal_F\rightarrow \bbr$ be a positive linear functional such that $\rho_{F^c}'(\Theta(\xbf))=e^{l_F(\xbf)}$, where (as before) $F^c:=\Delta\setminus F$. In particular, $l_{E^c}(\xbf)=\sum_{\alpha\in E^c} m_{\alpha} x_{\alpha}$.
	\item Let $l_{\chi}:\vcal_{E^c}\rightarrow \bbr$ be a positive linear functional such that $\chi(\Theta(\xbf))=e^{l_{\chi}(\xbf)}$. So in the above notation we have
	$l_{\chi}(\xbf)=\sum_{\alpha\in E^c} c_{\alpha} x_{\alpha}$.
	\item In this Section (as we said earlier) $E$ and $\chi$ are fixed. For $T\in \bbr^{\ge 1}$ and $\ybf\in \vcal_{E^c}$ let
	\[
	\vcal^+_{\ybf,T}:=\{\xbf\in \vcal_{E^c}|\h \xbf-\ybf\in \vcal^+_{E^c},\h l_{\chi}(\xbf)\le \log T\}.
	\]
 Notice that $\vcal^+_{\ybf,T}=\vcal^+_{0,e^{l_{\chi}(-\ybf)}T}+\ybf$.
\end{enumerate}
\end{definition}   
\begin{lem}\label{l:MaxCone}
For $F\subseteq \Delta$, let $l_1$ and $l_2$ be two positive linear functional on $\vcal_F$, and let $\mathbf{P}(\vcal_F^+):=\{[\vbf]|\h \vbf\in \vcal_F^+\}$ where $[\vbf]:=\bbr^+ \vbf$. Then
\begin{enumerate}
	\item The function $g:\mathbf{P}(\vcal_F^+)\rightarrow \bbr^+$, $g([\vbf]):=l_1(\vbf)/l_2(\vbf)$ is well-defined and continuous.
	\item We have 
	\[
	\max_{[\vbf]\in \mathbf{P}(\vcal_F^+)}g([\vbf])=\max_{\alpha\in F} g([\ebf_{\alpha}]),
	\]
	 where $\{\ebf_{\alpha}\}_{\alpha\in\Delta}$ is the standard basis of $\bbr^{|\Delta|}$, i.e. for any $\alpha'\in\Delta\setminus\{\alpha\}$ the $\alpha'$-component of $\ebf_{\alpha}$ is zero and the $\alpha$-component is 1.  
\end{enumerate}
\end{lem}
\begin{proof}
The first part is clear. Hence $a:=\max_{[\vbf]\in \mathbf{P}(\vcal_F^+)}g([\vbf])$ exists. Let 
$
\vbf\in \ker(l_1-al_2)\cap \vcal^+_F.
$
And assume that $\vbf=\sum_{\alpha\in F'}v_{\alpha} \ebf_{\alpha}$ for some $v_{\alpha}\in \bbr^+$ and $\varnothing\subsetneq F'\subseteq F$. Then
\begin{align*}
l_1(\vbf)&=l_1(\sum_{\alpha\in F'} v_{\alpha}\ebf_{\alpha})=\sum_{\alpha\in F'} v_{\alpha} l_1(\ebf_{\alpha})\\
&\le a \sum_{\alpha\in F'} v_{\alpha} l_2(\ebf_{\alpha})=a l_2(\vbf)\\
&=l_1(\vbf).
\end{align*}
Therefore for any $\alpha\in F'$ we have that $a=g([\ebf_{\alpha}])$.
\end{proof}
\begin{definition}\label{d:aAndbParameters}
Let $E$ and $\chi$ be as before, and let $E\subseteq F\subseteq \Delta$. We set 
\[
\apzc_F:=\max_{[\vbf]\in \mathbf{P}(\vcal_{F\setminus E})}\frac{l_{E^c}(\vbf)}{l_{\chi}(\vbf)},
\]
and $\apzc:=\apzc_{\Delta}$. By Lemma~\ref{l:MaxCone}, we have 
\[
\apzc_F=\max_{\alpha\in F\setminus E}\frac{m_{\alpha}}{c_{\alpha}}.
\]
We set also $\bpzc_F:=\dim \ker(l_{E^c}-\apzc_F l_{\chi})|_{\vcal_{F\setminus E}}$ and $\bpzc:=\bpzc_{\Delta}$. We set 
\[
F_{\chi}:=E\cup \{\alpha\in E^c|\h \apzc=m_{\alpha}/c_{\alpha}\}
\]
and call it the {\em max-type} of $\chi$. In particular, by Lemma~\ref{l:MaxCone}, 
we have $\apzc=\apzc_{F_{\chi}}$ and $\bpzc=|F_{\chi}\setminus E|$.
\end{definition}
\begin{lem}~\label{int-est1} 
Let $E$ and $\chi$ be as before, and $E\subseteq F\subseteq \Delta$. Then 
\be\label{e:IntegralEstimate}
\int_{\xbf_{F\setminus E} \in\vcal^+_{0,T}\cap \vcal_{F\setminus E}} e^{l_{E^c}(\xbf_{F\setminus E})} d\xbf_{F\setminus E}\sim C\h T^{\apzc_F} (\log T)^{\bpzc_F-1},\hspace{3mm} \mbox{as T goes to infinity,}
\ee
where $\apzc_F,\bpzc_F$ are as in Definition \ref{d:aAndbParameters} and $C=C(\chi,F)$ is a positive number. In particular, 
\begin{enumerate}
	\item if the max-type $F_{\chi}$ of $\chi$ is not a subset of  $F$ (see Definition~\ref{d:aAndbParameters}), then 
	\[
	\lim_{T\rightarrow \infty} \frac{\int_{\xbf_{F\setminus E} \in\vcal^+_{0,T}\cap \vcal_{F\setminus E}} e^{l_{E^c}(\xbf_{F\setminus E})} d\xbf_{F\setminus E}}{f(T)}=0.
	\]
	\item if $\chi=\rho_E'$, then 
\[
\int_{\xbf\in\vcal^+_{0,T}} e^{l_{E^c}(\xbf)} d\xbf\sim C\h T (\log T)^{|E^c|-1},\hspace{3mm} \mbox{as T goes to infinity.}
\]
\end{enumerate}
\end{lem}
\begin{proof} For the proof of (\ref{e:IntegralEstimate}) see either~\cite{GW} or~\cite[Section 6]{GOS}. If $F_{\chi}\not\subseteq F$, then either $\apzc_F< \apzc$ or $\bpzc_F<\bpzc$. So by (\ref{e:IntegralEstimate}), one gets the first part. The second part is clear.
\end{proof}
\begin{cor}~\label{int-est2} Let $E$ and $\chi$ be as before. Let $\ybf\in \vcal_{E^c}$. Then for any $E\subseteq F\subseteq \Delta$ we have 
\[
\int_{\xbf\in \wcal_T} e^{l_{E^c}(\xbf)} d{\xbf}\sim Ce^{l_{E^c}(\ybf)-\apzc_F l_{\chi}(\ybf)}\h T^{\apzc_F} (\log T)^{\bpzc_F-1},\h\h\h\h\text{as $T$ goes to infinity,}
\]
where 
\[
\wcal_T:=\{\xbf\in \vcal|\h \xbf-\ybf\in \vcal^+_{F\setminus E},\h l_{\chi}(\xbf)\le \log T\}
\]
and $\apzc_F$, $\bpzc_F$ are as in Definition~\ref{d:aAndbParameters} 
and $C=C(\chi,E)$ is as in Lemma~\ref{int-est1}. 
\end{cor}
\begin{proof}
Let $\xbf':=\xbf-\ybf$. Then
\[
\int_{\xbf\in \wcal_T} e^{l_{E^c}(\xbf)} d\xbf 
=
\int_{\xbf'\in\vcal^+_{0,e^{-l_{\chi}(\ybf)}T}\cap \vcal_{F\setminus E}} e^{l_{E^c}(\xbf'+\ybf)} d\xbf'
=
e^{l_{E^c}(\ybf)} \int_{\xbf'\in\vcal^+_{0,e^{-l_{\chi}(\ybf)}T}\cap\vcal_{F\setminus E}} e^{l_{E^c}(\xbf')} d{\xbf'}.
\]
Therefore by Lemma~\ref{int-est1} we have
\begin{align*}
\int_{\xbf\in\wcal_T} e^{l_{E^c}(\xbf)} d\xbf &
\sim e^{l_{E^c}(\ybf)}\cdot C (e^{-l_{\chi}(\ybf)}T)^{\apzc_F}(-l_{\chi}(\ybf)+\log T)^{\bpzc_F-1}\\
&
\sim C e^{l_{E^c}(\ybf)-\apzc_F l_{\chi}(\ybf)} T^{\apzc_F}(\log T)^{\bpzc_F-1},
\end{align*}
as $T$ goes to infinity.
\end{proof}
Let $\{\Psi_{\vare}\}$ be a family of continuous nonnegative functions on $G/\Gamma$ 
which approximates the Dirac function at $\pi(e)\in G/\Gamma$, i.e. 
for any continuous function $f\in C(G/\Gamma)$ we have 
\[
\lim_{\vare\rightarrow 0}\langle \Psi_{\vare},f\rangle=f(\pi(e)).
\]  
Moreover assume that for every $F\subseteq\Delta$ we have 
\be\label{e;supp-psi-vare}
\mbox{supp}(\Psi_{\vare})\subseteq\pi(\{g\in G\h|\h \|\vartheta(g^{-1})\|\le e^{\vare} \})\cap \pi(K \Theta(B_{F^c,\vare}) Q_F),
\ee
where $\vartheta=\oplus_{\alpha\in\Delta} \wedge^{\dim R_u(\bbp_{\alpha})}\Ad$, $\Theta$ is as in (\ref{e:Theta}) and 
\[
B_{F^c,\vare}:=\{\xbf\in \vcal_{F^c}|\h\forall\h\alpha\in\Delta,\h |x_{\alpha}|\le \vare\}.
\]
 
Unlike in the proof of the geometric example, $\langle \wt{F}_T,\Psi_{\vare}\rangle$ 
does not converge to a constant function as $T$ goes to infinity. 
Nevertheless we can show that it does converge to a function $L$ of $\vare$.

\begin{prop}\label{p:AlmostWeak}
There are a positive real number $\vare_0$ and a function $L:(0,\vare_0)\rightarrow \bbr$ such that
\begin{enumerate}
	\item $\lim_{T\rightarrow \infty} \langle \wt{F}_T,\Psi_{\vare}\rangle=L(\vare)$ for any $0<\vare<\vare_0$.
	\item $\lim_{\vare\rightarrow 0^+}L(\vare)$ exists.
\end{enumerate} 
\end{prop} 
Proposition~\ref{p:AlmostWeak} plays a central role in this section and we prove it in several steps. 
Before doing so, let us prove Theorem~\ref{t:Flag} modulo Proposition~\ref{p:AlmostWeak}.

\begin{proof}[Proof of Theorem~\ref{t:Flag} module Proposition~\ref{p:AlmostWeak}.]
The proof is similar to the end of the proof in Section~\ref{s:CountingHorosphers}.
Let $\Psi_{\vare}$ be as above then in view of~\eqref{e;supp-psi-vare}
we have
\begin{equation}~\label{pointwise-flag}
\frac{f(e^{-\vare}T)}{f( T)}\cdot\widetilde{F}_{e^{-\vare}T}(\pi(e))\le
\langle \widetilde{F}_T,\Psi_{\vare}\rangle\le
\frac{f(e^{\vare}T)}{f(T)}\cdot\widetilde{F}_{e^{\vare}T}(\pi(e)).
\end{equation}
Now since for a given $\vare$ we have $\lim_{T\rightarrow \infty}\frac{f(e^{\pm\vare}T)}{f(T)}=e^{\pm O(\vare)},$ Theorem~\ref{t:Flag} follows from Proposition~\ref{p:AlmostWeak}.
\end{proof}

\begin{lem}\label{l:PartitionForIntegration}
Let $\chi$, $E$ and $\{\Psi_{\vare}\}$ be as before. Let $\vare$ be a (small) positive real number. For any (small) $\vare'>0$ there is a partition $\{\Rcal_{F}\}_{E\subseteq F\subseteq \Delta}$ of $\prod_{\alpha\in E}\{0\}\cdot\prod_{\alpha\in E^c} [-\vare,\infty)$ with the following properties (as before the implied constants depend only on $\bbg$ and $\chi$).
\begin{enumerate}
	\item We have $|(\Theta(\xbf)\mu_{Q_E})(\Psi_{\vare})-(\Theta({\rm pr}_{F^c}(\xbf))\mu_{Q_F})(\Psi_{\vare})|\le \vare'$ if $\xbf\in \Rcal_F$. 
	\item We have $|(\Theta(\xbf)\mu_{\peone})(\Psi_{\vare})-(\Theta({\rm pr}_{F^c}(\xbf))\mu_{\pfone})(\Psi_{\vare})|\le \vare'$ if $\xbf\in \Rcal_F$. 
	\item For any $E\subseteq F\subseteq \Delta$, $\Rcal_F=({\rm pr}_{F^c} \Rcal)\oplus (\ybf_F+\vcal^+_{F\setminus E})$ where $\ybf_F=y_F\sum_{\alpha\in F\setminus E}\ebf_{\alpha}$ for some positive number $y_F$. 
	\item For any $E\subseteq F\subseteq \Delta$, ${\rm pr}_{F^c}(\Rcal_F)$ is bounded. 
\end{enumerate}
\end{lem}
\begin{proof}
Parts (1), (3) and (4) are just rewriting Corollary~\ref{c:EPartitionIntoRegions} in the new setting,
we now show part (2). 
Recall from the beginning of this section that $P_\bullet^{(1)}=Y_\bullet A_\bullet''Q_\bullet $
where the product is an almost direct product, $Y_\bullet$ is a
compact group and $A_\bullet''\subseteq A'.$ Moreover, both $Y_\bullet$
and $A_\bullet''$ centralize $A.$ 
Also recall that if $E\subseteq F,$ then $Y_FA''_F\subseteq Y_EA_E.$ 

Now using~\eqref{e;Fubini} we have
\begin{align*}
|\int_{G/\Gamma}\Psi_\vare d(\Theta(\xbf)\mu_{\peone})&-\int_{G/\Gamma}\Psi_\vare d(\Theta({\rm pr}_{F^c}(\xbf))\mu_{\pfone})|=\\
&|\int_{G/\Gamma}\Psi_\vare d\mu_{Y_{E}}d\mu_{A_E''}d(\Theta(\xbf)\mu_{Q_E})-\int_{G/\Gamma}\Psi_\vare 
d\mu_{Y_{F}}d\mu_{A''_F}d(\Theta({\rm pr}_{F^c}(\xbf))\mu_{Q_F})|\\
&\int_{G/\Gamma}\int_{G/\Gamma}
|(\Theta(\xbf)\mu_{Q_E})(\Psi_{\vare})-(\Theta({\rm pr}_{F^c}(\xbf))\mu_{Q_F})(\Psi_{\vare})|
d\mu_{Y_{F}}d\mu_{A''_F}d\nu.
\end{align*}
Where $d\nu$ is the probability measure
on $(Y_E/Y_E\cap\Gamma)/(Y_F/Y_F\cap\Gamma)\cdot(A''_E/A''_E\cap\Gamma)/(A''_F/A''_F\cap\Gamma);$ 
recall that $Y_FA''_F\subseteq Y_EA''_E$ is a normal subgroup.

The conclusion in (2) now follows from (1).
\end{proof}

The next lemma helps us to see the importance of the max-type $F_{\chi}$ of $\chi$.
\begin{lem}\label{l:GettingToMaxTypeTerm}
Let $\chi$, $E$ and $\{\Psi_{\vare}\}$ be as before. Let $\vare,\vare'$ be (small) positive numbers. Let $\{\Rcal_F\}_{E\subseteq F\subseteq \Delta}$ and $\{\ybf_F\}_{E\subseteq F\subseteq \Delta}$ be as in Lemma~\ref{l:PartitionForIntegration} and further we can and will assume that $y_F$ is large depending on $\vare$ and $\vare'$. Then there is a positive number $T_0:=T_0(\vare',\vare)$ such that for any $T\ge T_0$ and any $\xbf_{F^c}\in {\rm pr}_{F^c}(\Rcal_F)$, where $F\neq F_{\chi}$, we have
\[
g_T(\xbf_{F^c}):=\frac{\int_{\Rcal_{F,T}(\xbf_{F^c})}e^{l_{E^c}(\xbf_F)}d\xbf}{f(T)}\ll_{\vare} \vare'
\] 
where $\Rcal_{F,T}(\xbf_{F^c}):=\{\xbf\in (\xbf_{F^c}+\ybf_F)+\vcal^+_{F\setminus E}|\h l_{\chi}(\xbf)\le \log T\}$ and as before 
\[
f(T)=\int_{\xbf\in \vcal^+_{0,T}} e^{l_{E^c}(\xbf)}d\xbf
,\]
(See Figure~\ref{F:Integration}).
\end{lem}
\begin{figure}
\includegraphics[width=\columnwidth]{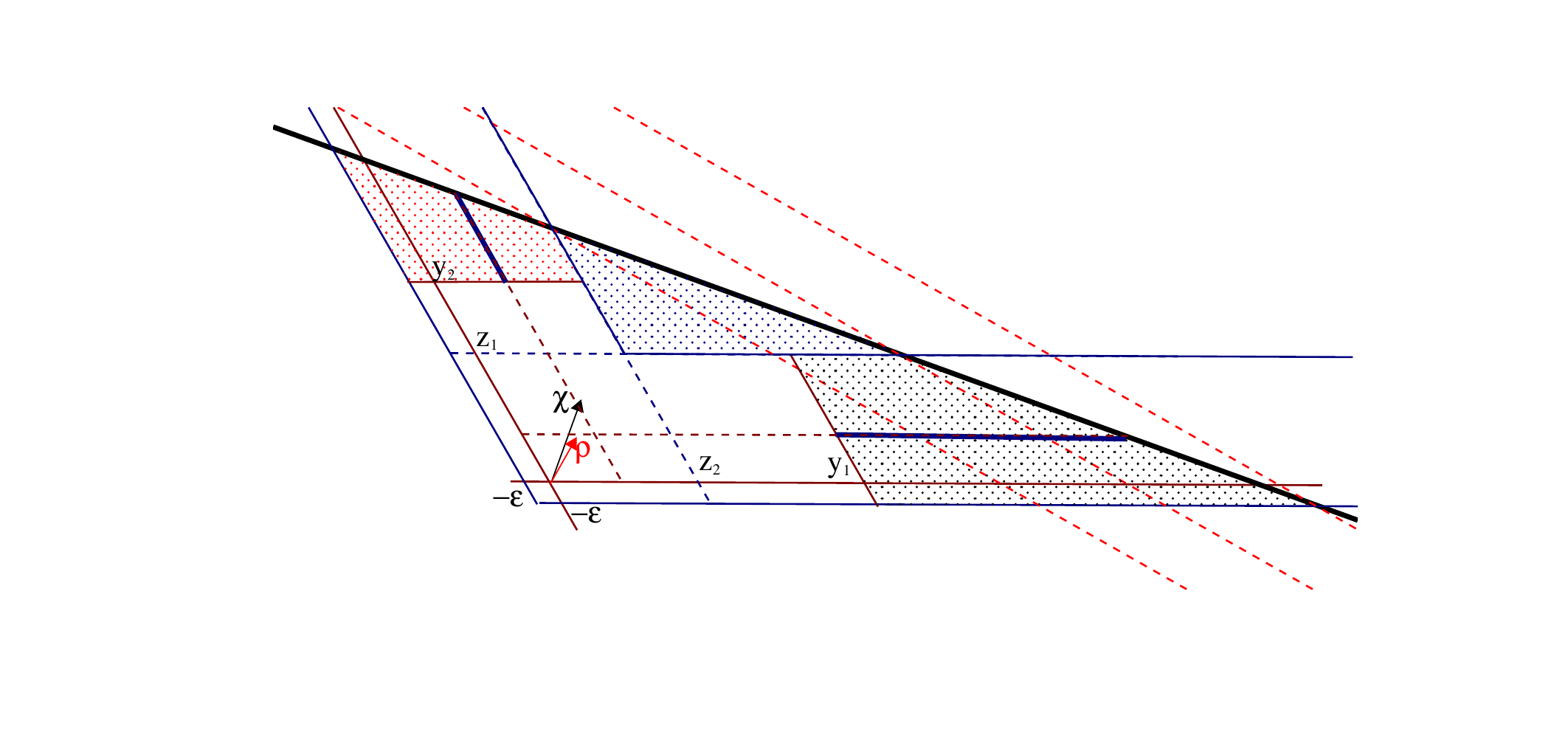}
\caption{Integration regions, $\mathbb{SL}_3$ case.}
\label{F:Integration}
\end{figure}
\begin{proof}
By the proof of Corollary~\ref{int-est2}, we have that 
\[
\int_{\Rcal_{F,T}(\xbf_{F^c})}e^{l_{E^c}(\xbf_F)}d\xbf \sim C e^{l_{E^c}(\ybf_F)-\apzc_F l_{\chi}(\xbf_{F^c}+\ybf_F)} T^{\apzc_F} (\log T)^{\bpzc_F-1},
\]
as $T$ goes to infinity. Hence there is nothing to prove if either $\apzc>\apzc_F$ or $\bpzc>\bpzc_F$. So without loss of generality we can and will assume that $F_{\chi}\subseteq F$, and so $\apzc=\apzc_F$ and $\bpzc=\bpzc_F$. 
Now since 
\[
\xbf_{F^c}\in \prod_{\alpha\in F}\{0\}\cdot \prod_{\alpha\in \Delta\setminus F} [-\vare,\infty),
\]
and $\apzc>0$ we have that 
$
e^{-\apzc l_{\chi}(\xbf)}\ll_{\vare} 1.
$

If $F_{\chi}\subsetneq F$, then $l_{E^c}(\sum_{\alpha\in F\setminus E}\ebf_{\alpha})-\apzc l_{\chi}(\sum_{\alpha\in F\setminus E}\ebf_{\alpha})\le \max_{\alpha\in \Delta\setminus F_{\chi}}l_{E^c}(\ebf_{\alpha})-\apzc l_{\chi}(\ebf_{\alpha})<0$. So assuming $y_F$ is large enough, we have that 
\[
e^{l_{E^c}(\ybf_F)-\apzc l_{\chi}(\ybf_F)}\ll \vare'.
\]
Altogether we get the desired result.
\end{proof}
\begin{remark}
It is worth pointing out that 
\[
\Rcal_{F,T}(\xbf_{F^c})=\{\xbf' \in \Rcal_F|\h {\rm pr}_{F^c}(\xbf')=\xbf_{F^c},\h l_{\chi}(\xbf')\le \log T\},
\]
for large enough $T$ (depending on $\vare$, $\vare'$ and the choice of $\{y_F\}_F$).
\end{remark}
\begin{lem}\label{l:FirstStep}
Let $\chi$, $E$ and $\{\Psi_{\vare}\}$ be as before. Let $\vare$ be a (small) positive real number. 
For a given (small) $\vare'>0$, let $\Rcal_F$ be as in Lemma~\ref{l:PartitionForIntegration}. 
Then there is $T_1=T_1(\vare,\vare')$ such that for $T\ge T_1$ we have
\[
\langle \wt{F}_T,\Psi_{\vare}\rangle=C\int_K\int_{\xbf\in {\rm pr}_{F_{\chi}^c}(\Rcal_{F_{\chi}})} (k\Theta(\xbf)\mu_{Q_{F_{\chi}}})(\Psi_{\vare}) e^{l_{E^c}(\xbf)-\apzc l_{\chi}(\xbf)} d\xbf  dk+O_{\vare}(\vare'),
\] 
where $C$ is a positive number given in~\ref{int-est1}, $F_{\chi}$ is the max-type of 
$\chi$ (see Definition~\ref{d:aAndbParameters}). We use this convention that, if $F_{\chi}=\Delta$, 
then the inner integral is one, and so 
\[
\langle \wt{F}_T,\Psi_{\vare}\rangle=C\vol(K)+O_{\vare}(\vare').
\]
\end{lem}
\begin{proof}
First we notice that by Lemma~\ref{out}, if $\lambda_\alpha(a)<e^{-\vare}$ for some $\alpha,$ then $\int_{G/\Gamma} \Psi_{\vare} d(a\mu_{Q_E})=0$. 
For any real number $r$, let $\ccal_{T,r}:=\{\xbf\in r\sum_{\alpha\in E^c}\ebf_{\alpha}+\vcal^+|\h l_{\chi}(\xbf)\le \log T\}$ and 
\[
\Rcal_{F,T}:=\Rcal_F\cap \ccal_{T,-\vare}=\{\xbf\in \Rcal_F|\h l_{\chi}(\xbf)\le \log T\}.
\]
Now using Lemma~\ref{B-decomposition} and the decomposition of 
the Haar measure, e.g.~\cite[Proposition 8.44]{Kn}, we have
\begin{align}
\langle \wt{F}_T,\Psi_{\vare}\rangle&
=\frac{1}{f(T)}\int_K\int_{\xbf\in \ccal_{T,-\vare}}\int_{G/\Gamma} \Psi_{\vare}(g'\Gamma) 
d(k\Theta(\xbf)\mu_{\peone})(g') e^{l_{E^c}(\xbf)} d\xbf dk
\\
\notag {}^{\text{(Lemma \ref{l:PartitionForIntegration})}\leadsto}&
=\sum_{E\subseteq F\subseteq \Delta}\frac{1}{f(T)}\int_K\int_{\Rcal_{F,T}}(k\Theta(\xbf)\mu_{\peone})(\Psi_{\vare}) e^{l_{E^c}(\xbf)} d\xbf dk
\\
\notag {}^{\text{(Lemma \ref{l:PartitionForIntegration})}\leadsto}&
= \sum_{E\subseteq F\subseteq \Delta}\frac{1}{f(T)}\int_K\int_{\Rcal_{F,T}}(k\Theta({\rm pr}_{F^c}(\xbf))\mu_{\pfone})(\Psi_{\vare}) e^{l_{E^c}(\xbf)} d\xbf dk 
\\
\notag
&+ \vare' \vol (K) O\left(\frac{\int_{\xbf\in \ccal_{T,-\vare}} e^{l_{E^c}(\xbf)} d\xbf}{f(T)}\right)
\end{align}
Hence by Corollary \ref{int-est2} we have
\begin{align}\label{e:pairing-flag}
\langle \wt{F}_T,\Psi_{\vare}\rangle&
=\sum_{E\subseteq F\subseteq \Delta}\frac{1}{f(T)}\int_K\int_{\Rcal_{F,T}}(k\Theta({\rm pr}_{F^c}(\xbf))\mu_{\pfone})(\Psi_{\vare}) e^{l_{E^c}(\xbf)} d\xbf dk +O_{\vare}(\vare')
\\
\notag {}^{\text{(Lemma~\ref{l:GettingToMaxTypeTerm}, $T\gg_{\vare,\vare'} 1$)}\leadsto}&
=\sum_{E\subseteq F\subseteq \Delta} \int_K\int_{{\rm pr}_{F^c}(\Rcal_F)} (k\Theta(\xbf_{F^c})\mu_{\pfone})(\Psi_{\vare}) e^{l_{E^c}(\xbf_{F^c})} g_T(\xbf_{F^c}) d\xbf_{F^c}  dk +O_{\vare}(\vare')
\\
\notag {}^{\text{(Lemma~\ref{l:GettingToMaxTypeTerm})}\leadsto}
&
=C\int_K\int_{\xbf\in {\rm pr}_{F_{\chi}^c}(\Rcal_{F_{\chi}})} (k\Theta(\xbf)\mu_{P_{F_{\chi}}^{(1)}})(\Psi_{\vare}) e^{l_{E^c}(\xbf)-\apzc l_{\chi}(\xbf)} d\xbf  dk
\\
\notag &
+O_{\vare}(\vare')\sum_{E\subseteq F\subseteq \Delta, F\neq F_\chi}\int_K\int_{{\rm pr}_{F^c}(\Rcal_F)} (k\Theta(\xbf_{F^c})
\mu_{\pfone})(\Psi_{\vare}) e^{l_{E^c}(\xbf_{F^c})} d\xbf_{F^c}  dk+O_{\vare}(\vare')
\\
\notag &
= C\int_K\int_{\xbf\in {\rm pr}_{F_{\chi}^c}(\Rcal_{F_{\chi}})} (k\Theta(\xbf)\mu_{P_{F_{\chi}}^{(1)}})(\Psi_{\vare}) e^{l_{E^c}(\xbf)-\apzc l_{\chi}(\xbf)} d\xbf  dk +O_{\vare}(\vare').
\end{align}
\end{proof}

\begin{cor}~\label{c:Anticanonical}
In the anti-canonical line-bundle case, i.e. when $\chi=\rho_E'$, we have that
\[
\lim_{T\rightarrow \infty} \langle \wt{F}_T,\Psi_{\vare}\rangle =C\vol (K),
\]
where $C$ is the constant given in Lemma~\ref{int-est1}. 

In particular, this implies Proposition~\ref{p:AlmostWeak} and therefore Theorem~\ref{t:Flag} when $\chi=\rho_E'$.
\end{cor}
\begin{proof}
In this case we have $F_{\chi}=\Delta$ and so by Lemma~\ref{l:FirstStep} and the fact that $\mu_G(\Psi_{\vare})=1$ we have that
\[
\langle \wt{F}_T,\Psi_{\vare}\rangle=C\vol (K)+O_{\vare}(\vare'),
\]
for large enough $T$. 
\end{proof}

The following Lemma is a well known fact and follows from reduction theory and 
results from harmonic analysis on $G,$ see ~\cite[Lemma 23]{HC}. We give a self-contained 
treatment of this convergence using a somewhat technique.

Let $E$ and $\chi$ be as in the beginning of this section. Recall that for any 
$E\subseteq F\subseteq \Delta;$ 
we put $\rho'_F=\sum_{\alpha\not\in F}m_\alpha\lambda_\alpha.$  

Let $\eta_{F_\chi}=\wedge ^{\dim R_u(\bbp_{F_\chi})}\Ad$ 
and let $v_{F_\chi}$ a unit vector on the (rational) line 
\[
\wedge^{\dim R_u(\bbp_{F_\chi})}\Lie(R_u(\bbp_{F_{\chi}})).
\]
Note that for any $p\in\bbp_{F_\chi}$ we have $\eta_{F_\chi}( p)v_{F_\chi}=\rho'_{F_\chi}( p)v_{F_\chi},$
in particular, ${\rm Stab}_{\mathbb{G}}(v_{F_\chi})=\bbp_{F_\chi}^{(1)}.$ 

Fix a bounded neighborhood of the identity $\mathcal{O}\subset G.$ 
By the Iwasawa decomposition any $g\in G$ can be decomposed as $g=k_g\Theta(H(g))q_g$ 
where $k_g\in K,$ $q_g\in\pfone$  and $H(g)\in \vcal_{F_{\chi}^c}$.

\begin{lem}\label{l:abs-conv}
The series 
\[
\xi(g\Gamma)=\sum_{{\gamma}\in \Gamma/\Gamma\cap P_{F_\chi}^{(1)}} e^{-\apzc l_\chi (H(g\gamma))}
\]
is uniformly convergent on $\mathcal{O},$ and in particular, it defines an analytic function 
on $\mathcal{O}.$ 
\end{lem}

\begin{proof} 
We need to show that the above series is uniformly convergent.
First note that from the definition of $H(g)$ it follows that 
$l_\chi(H(g))=\sum_{\alpha\not\in F_\chi}c_\alpha H(g)_{\alpha}.$

Now, since $\|\;\|$ is $K$-invariant and $\pfone$ 
fixes $v_{F_\chi}$ we have: $2^n\leq \|\eta_{F_\chi}(g\gamma) v_{F_\chi}\|< 2^{n+1}$ implies that 
\be\label{e;m-alpha-inner}
n\leq_\mathcal{O} l_{F_{\chi}^c}(H(g\gamma))=\sum_{\alpha\not\in F_{\chi}}m_{\alpha}H(g\gamma)_{\alpha}\leq_{\mathcal{O}} n+1,
\ee
where $g\gamma=k_{g\gamma}\Theta(H(g\gamma))q_{g\gamma}$ 
as above. 

We utilize the notation from Section~\ref{ss:MargulisDani}.
As was discussed in the proof of Lemma~\ref{out}, see also Lemma~\ref{root},
there is a constant $D_{\mathcal{O}}\geq 0$ so that 
note that for any $\alpha\in\Delta$ and every $g\in \mathcal{O}$
we have $\lambda_\alpha(g\gamma)\geq e^{-D_{\mathcal{O}}}.$

We fix a positive number $\delta$ such that $m_\alpha(1+\delta)<\apzc c_\alpha$ 
for any $\alpha\not\in F_\chi$. 
Now for any $g\in\mathcal{O}$ and $\gamma\in \Gamma/\Gamma\cap \pfone$ 
so that $e^n\leq \|\eta_{F_\chi}(g\gamma) v_{F_\chi}\|< e^{n+1}$
we have 
\begin{align*}
(\apzc c_\alpha)\cdot H(g\gamma)&=
\sum_{-D_{\mathcal{O}}\leq H(g\gamma)_\alpha<0} \apzc c_\alpha H(g\gamma)_\alpha+
\sum_{ H(g\gamma)_\alpha\geq0} \apzc c_\alpha H(g\gamma)_\alpha\\
&\geq D'(\mathcal{O},\chi)+\sum_{ H(g\gamma)_\alpha\geq0}(1+\delta)(m_\alpha)H(g\gamma)_\alpha.
\end{align*}
This, a similar calculation for $(m_\alpha)\cdot(H(g\gamma))$ 
and~\eqref{e;m-alpha-inner} imply that for large enough $n$
we have
\[
e^{-\apzc l_\chi (H(g\gamma))}\leq_{\mathcal{O,\chi}} e^{(-1-\delta)n}.
\]
For any $g\in\mathcal{O}$ and any $T'>0$, let 
$
S_{g,T'}=\{\gamma(\Gamma\cap \pfone)\h|\h \|\eta_{F_\chi}(g\gamma) v_{F_\chi}\|\le T'\}.
$ 
Using the anticanonical case for the parabolic $\bbp_{F_\chi}$, see Corollary~\ref{c:Anticanonical}, 
we have 
\[
 \mbox{$|S_{g,T'}|\sim T' \log(T')^{|F_\chi^c|-1},\quad$ 
uniformly for $g\in\mathcal{O}.$}
\]
Hence for large enough $N$'s we have
\[
\sum^{\infty}_{n=N} e^{(-1-\delta)n} |S_{e^{n+1}}| \le 
2\sum^{\infty}_{n=N} e^{(-1-\delta)n} e^{n+1} (n+1)^{|F_\chi^c|-1} <\infty,
\]
this finishes the proof.
\end{proof}

\begin{proof}[Proof of of Proposition~\ref{p:AlmostWeak}.] 
Let the notation be as above. 
Put 
\[
\mathcal{C}_{-\vare}(F_\chi)=\{{\bf x}\in\bbr^{|F_{\chi}^c|}\h|
-\vare\le x_\alpha, \forall \alpha \in F_{\chi}^c\}.
\]
Using the decomposition of the Haar measure we have
\begin{align*}
\int_K\int_{\mathcal{C}_{-\vare}(F_\chi)} (k\Theta(\xbf)\mu_{P_{F_{\chi}}^{(1)}})(\Psi_{\vare})
e^{l_{E^c}(\xbf)-\apzc l_{\chi}(\xbf)} d\xbf  dk
&=\int_{G/P_{F_\chi}^{(1)}}\int_{P_{F_\chi}^{(1)}/P_{F_\chi}^{(1)}\cap\Gamma} \Psi_{\vare}(g'q)\mu_{P_{F_{\chi}}^{(1)}}( q)e^{-\apzc l_{\chi}(H(g'))} dg'\\
&=\int_{G/P_{F_\chi}^{(1)}\cap\Gamma}\Psi_{\vare}(g')e^{-\apzc l_{\chi}(H(g'))} dg'\\
&=\int_{G/\Gamma}\Psi_{\vare}(g')\sum_{\Gamma/\Gamma\cap P_{F_\chi}^{(1)}}e^{-\apzc l_{\chi}(H(g'))} dg'=\langle \Psi_{\vare},\xi\rangle.
\end{align*}
So by Lemma~\ref{l:abs-conv}, we have $L(\vare):=\langle \Psi_{\vare},\xi\rangle$ is defined for $\vare<\vare_0$
and $\lim_{\vare\to 0} L(\vare)$ exists. 

Note that ${\rm pr}_{F_{\chi}^c}(\Rcal_{F_{\chi}})$ is a bounded region 
independent of $T.$ 
Moreover, as $\vare'\to0$ we have ${\rm pr}_{F_{\chi}^c}(\Rcal_{F_{\chi}})$
covers $\mathcal{C}_{-\vare}(F_\chi)$
Therefore, by if we let $T\to\infty$ in Lemma~\ref{int-est1}, 
see in particular~\eqref{e:pairing-flag} we get
\[
\lim_{T\to\infty}\langle \wt{F}_T,\Psi_{\vare}\rangle=C\int_K\int_{\mathcal{C}_{-\vare}(F_\chi)} (k\Theta(\xbf)\mu_{P_{F_{\chi}}^{(1)}})(\Psi_{\vare}) e^{l_{E^c}(\xbf)} e^{l_{E^c}(\xbf)-\apzc l_{\chi}(\xbf)} d\xbf  dk.
\]
This and the above finish the proof.
\end{proof}
\section*{Acknowledgment} 
We would like to thank Gregory Margulis and Peter Sarnak for the helpful conversations. We are grateful to Zeev Rudnick and Nicholas Templier for asking us about the absolutely convergence case. We also thank the anonymous referee whose comments helped to improved the exposition. 

\end{document}